\documentclass[a4paper, oneside,11pt]{amsart}
\usepackage{amsmath}
\usepackage{amssymb,amsbsy,amsmath,amsfonts,amssymb,amscd,amsthm}
\usepackage{stmaryrd}
\usepackage{mathrsfs}
\usepackage{bbm}
\usepackage{amssymb}
\usepackage{graphicx}
\usepackage{color}

\usepackage[dvipsnames]{xcolor}
\usepackage[colorlinks,linkcolor={blue},citecolor={blue},urlcolor={purple},]{hyperref}

\def\N{{\mathbb N}}
\def\Z{{\mathbb Z}}

\def\R{{\mathbb R}}
\def\C{{\mathbb C}}
\def\T{{\mathbb T}}

\def\P{{\mathbb P}}

\theoremstyle{plain}
\newtheorem{theorem}{Theorem}[section]
\theoremstyle{remark}

\newtheorem{example}[theorem]{Example}
\theoremstyle{plain}
\newtheorem{corollary}[theorem]{Corollary}
\newtheorem{lemma}[theorem]{Lemma}
\newtheorem{proposition}[theorem]{Proposition}
\newtheorem{definition}[theorem]{Definition}

\newtheorem{problem}[theorem]{Problem}
\numberwithin{equation}{section}

\newcommand{\norm}[1]{\lVert #1 \rVert}

\newcommand{\nrms}[1]{\Bigl\|#1\Bigr\|}

\newcommand{\has}[1]{\Bigl(#1\Bigr)}
\newcommand{\cbraceb}[1]{\bigl\{#1\bigr\}}

\newcommand{\ee}{\mathrm{e}}
\newcommand{\dd}{\hspace{2pt}\mathrm{d}}

\newcommand{\one}{{{\bf 1}}}

\def\O{\Omega}

\def\cI{{\mathcal I}}

\def\calL{{\mathcal L}}

\def\cP{{\mathscr P}}

\DeclareMathOperator {\UMD}{UMD}

\DeclareMathOperator{\sgn}{sgn}
 \numberwithin{equation}{section}

\newcommand{\lb}{\langle}
\newcommand{\rb}{\rangle}

\DeclareFontFamily{U}{mathx}{\hyphenchar\font45}
\DeclareFontShape{U}{mathx}{m}{n}{<5> <6> <7> <8> <9> <10> <10.95> <12> <14.4> <17.28> <20.74> <24.88> mathx10}{}
\DeclareSymbolFont{mathx}{U}{mathx}{m}{n}
\DeclareFontSubstitution{U}{mathx}{m}{n}
\DeclareMathAccent{\widecheck}{0}{mathx}{"71}

\usepackage{color}
\usepackage{xcolor}

\usepackage{caption}
\usepackage{subcaption}

\usepackage[shortlabels]{enumitem}

\setlist[enumerate]{leftmargin=20pt}
\setlist[itemize]{leftmargin=20pt}

\allowdisplaybreaks

\begin{document}
	\title{Strongly Kreiss Bounded Operators in $\UMD$ Banach Spaces}

\author[C. Deng]{Chenxi Deng}
\address[C. Deng]{School of Mathematics and Statistics\\
Beijing Institute of Technology \\ Beijing 100081\\China }
\email{chenxideng@bit.edu.cn}

\author[E. Lorist]{Emiel Lorist}
\address[E. Lorist]{Delft Institute of Applied Mathematics\\
Delft University of Technology \\ P.O. Box 5031\\ 2600 GA Delft\\The
Netherlands}
\email{e.lorist@tudelft.nl}

\author[M.C. Veraar]{Mark Veraar}
\address[M.C. Veraar]{Delft Institute of Applied Mathematics\\
Delft University of Technology \\ P.O. Box 5031\\ 2600 GA Delft\\The
Netherlands}
\email{m.c.veraar@tudelft.nl}

\thanks{The first author is supported by China Scholarship Council (CSC). The third author is supported by the VICI subsidy VI.C.212.027 of the Netherlands Organisation for Scientific Research (NWO)}
	
\begin{abstract}
In this paper we give growth estimates for $\|T^n\|$ for $n\to \infty$ in the case $T$ is a strongly Kreiss bounded operator on a $\UMD$ Banach space $X$.  In several special cases we provide explicit growth rates. This includes known cases such as Hilbert and $L^p$-spaces, but also intermediate $\UMD$ spaces such as non-commutative $L^p$-spaces and variable Lebesgue spaces. \end{abstract}
	
\keywords{(strongly) Kreiss bounded, power boundedness, UMD space, discrete semigroup, Fourier multipliers, decompositions}
\subjclass[2020]{Primary: 47D03, Secondary: 47A30, 47A35, 42A61}
\maketitle


\section{Introduction}
Let $X$ be a Banach space. Suppose that $T\in \mathcal L(X)$ is power bounded, i.e., there exists a constant $C\geq 0$ such that $\|T^n\|\leq C$ for all $n\geq 1$. If $T$ is power bounded, then by the Neumann series
\begin{align*}
(\lambda-T)^{-1} &=\sum_{n\geq 0} \frac{T^n}{\lambda^{n+1}},   &|\lambda|>1,\intertext{one obtains}
\|(\lambda-T)^{-1}\| &\leq \sum_{n\geq 0} \frac{\|T^n\|}{|\lambda|^{n+1}}\leq \frac{C}{|\lambda|-1},  &|\lambda|>1.
\end{align*}
One can repeat the above calculation after differentiation, to see that
\begin{align*}
\|(\lambda-T)^{-k}\| \leq \frac{C}{(|\lambda|-1)^{k}}, \qquad |\lambda|>1,  \, k\in \N.
\end{align*}

These observations motivate the following definitions. An operator $T\in \mathcal{L}(X)$ is called {\em Kreiss bounded with constant $K$} if
\begin{align}\label{eq:Kreiss}
\|(\lambda-T)^{-1}\|\leq  \frac{K}{|\lambda|-1}, \qquad |\lambda|>1,
\end{align}
and $T$ is called {\em strongly Kreiss bounded with constant $K_s$} if
\begin{align}\label{k0}
\|(\lambda-T)^{-n}\|\leq \frac{K_s}{(|\lambda|-1)^{n}}, \qquad  |\lambda|>1,  \ n\in \N.
\end{align}
By the above observations, any power bounded operator is (strongly) Kreiss bounded. By letting $\lambda\to \infty$ one sees that $K_s,K\geq 1$.

 In applications to numerics and ergodic theory, one often needs power boundedness of $T$ or sharp estimates for $\|T^n\|$ as $n\to \infty$, which can be difficult to obtain directly. However, it is often possible to check (strong) Kreiss boundedness. Therefore, it is useful to investigate the converse to the above observations:
 \begin{enumerate}[(i)]
   \item\label{it:q1} Does {(strong)} Kreiss boundedness imply power boundedness?
   \item\label{it:q2} If this is not the case, which growth of $\|T^n\|$ can one obtain from the (strong) Kreiss boundedness of $T$?
 \end{enumerate}

In the  continuous time setting, the Hille-Yosida theorem provides a result of this form. It gives the equivalent characterization between the generation of bounded $C_0$-semigroups and the powers of the resolvent of its generator. Moreover, the Hille-Yosida theorem yields that $T$ is  strongly Kreiss bounded with constant $K_s$ if and only if
      \begin{align}\label{k1}
     \|\ee^{\xi T}\|\leq K_s \,\ee^{|\xi|},\qquad \xi\in \mathbb{C},
  \end{align}
see \cite[Proposition 1.1]{Nevanlinna}.

There is a gap between \eqref{k1} and power boundedness of $T$, stemming from the gap between the growth of an entire function and the decay of its Taylor coefficients ({see} \cite{Nevanlinna}). Therefore, the answer to Question \ref{it:q1} is unfortunately negative: not every (strongly) Kreiss bounded operator is power bounded. Counterexamples to this and related questions can be found in \cite{AC, CHMM,DKS, Kraa,LN}.

Question \ref{it:q2} has been extensively studied. For instance, using Cauchy's integral formula, one can check that if $T$ is Kreiss bounded with constant $K$, then (see \cite[p.9]{strikwerda}) we have
\begin{align}\label{eq:Kreisslinear}
\|T^n\|&\leq K \ee (n+1),  && n\in \N,
\intertext{
and, if $T$ is strongly Kreiss bounded with constant $K_s$, then (see  \cite[Theorem~2.1]{LN}) we have
}\label{eq:strongKreisslinear}
\|T^n\|&\leq K_s \sqrt{2\pi (n+1)}, &&n\in\N.
\end{align}
Moreover, these growth rates in $n$ are known to be optimal in general Banach spaces, see \cite{shields,strikwerda} for Kreiss bounded operators, and \cite[Example 2.2]{LN} for strongly Kreiss bounded operators.

Under geometric assumptions on $X$ one can improve the above bounds. In the special case that $X$ is $d$-dimensional, the ``Kreiss matrix theorem'' (see \cite{kreiss19621,leveque1984,Spijker}) asserts that Kreiss boundedness with constant $K$ implies $T$ is power bounded with $\|T^n\|\leq K \ee d$. In applications, the dimension may be very large (see \cite{DKS}), so it is of interest to study the sharpness with respect to $d$, which was  established in \cite{Kraa} up to multiplicative constants. In the finite dimensional setting, this seemed the end of the story. However, 20 years later in \cite{nikolski}, it was shown that the bound can be improved to sublinear growth in $d$  under further conditions.

In the infinite dimensional setting several results are known which improve the estimate \eqref{eq:Kreisslinear} for Kreiss bounded operators and the estimate  \eqref{eq:strongKreisslinear} for strongly Kreiss bounded operators:
\begin{itemize}
  \item If $X$ is a Hilbert space:
  \begin{itemize}
    \item \eqref{eq:Kreisslinear} can be improved to  $\|T^n\| = O\bigl({n}/{\sqrt{\log (n+2}}\bigr)$ (see  \cite[Theorem 5]{BM} and \cite[Theorem 4.1]{CCEL}).
        \item   \eqref{eq:strongKreisslinear} can be improved to $\|T^n\|= O((\log (n+2))^{\beta})$ for some $\beta>0$ depending on $T$ (see \cite[Theorem 4.5]{CCEL}). Moreover it is also shown in \cite[Proposition 4.9]{CCEL} that $\beta$ can be arbitrary large.
  \end{itemize}
    \item If $X = L^p$ with $p\in (1, \infty)\setminus\{2\}$:
    \begin{itemize}
      \item  \eqref{eq:Kreisslinear} can be improved to  $\|T^n\| = O\big({n}/{\sqrt{\log (n+2)}}\big)$ as well (see \cite[Corollary 3.2]{cuny}).
      \item  \eqref{eq:strongKreisslinear} can be improved to  $\|T^n\|= O(n^{|\frac{1}{2}-\frac{1}{p}|}(\log (n+2))^{\beta})$ for some $\beta>1$, where the number $|\frac{1}{2}-\frac{1}{p}|$ is optimal (see \cite[Theorem 1.1]{AC}).
    \end{itemize}
    \item If $X$ is a $\UMD$ space,  $q$ and $q^*$ denote the (finite) cotypes of $X$ and $X^*$, respectively:
        \begin{itemize}
          \item \eqref{eq:Kreisslinear} can be improved to $\|T^n\|=O(n/(\log (n+2))^{\beta})$ with $\beta=\frac{1}{q\wedge q^*}$ (see \cite[Theorem 3.1]{cuny}).
        \end{itemize}
\end{itemize}
\begin{table}[h]\footnotesize
\caption{\small Growth rates for (strongly) Kreiss bounded operators in various spaces.}\label{table}
\begin{tabular}{|c|c|c|c|c|}
  \hline \rule[0pt]{0pt}{10pt}
 &  Banach & Hilbert & $L^p$ & UMD \\
 \hline\rule[0pt]{0pt}{12pt}
KB&  $O(n)$ & $O\big({n/\sqrt{\log (n+2)}}\big)$  & $O({n/\sqrt{\log (n+2)}})$ & $O({n/(\log (n+2))^{\beta}})$ \\ [0.2ex]
 \hline\rule[0pt]{0pt}{12pt}
SKB  & $ O(\sqrt{n})$ & $ O((\log (n+2)))^{\beta})$  & $O(n^{|\frac{1}{2}-\frac{1}{p}|}(\log (n+2))^{\beta})$ & This paper\\
 [0.2ex]
 \hline
\end{tabular}
\end{table}
See Table \ref{table} for an overview of these results.
An improvement of \eqref{eq:strongKreisslinear} for general UMD spaces seems to be missing. The main results of this paper give such improvements. Moreover, we recover the results for strongly Kreiss bounded operators from \cite[Theorem 4.5]{CCEL} and \cite[Theorem 1.1]{AC} in the Hilbert and $L^p$-cases, respectively. The following two results (see Corollaries \ref{cor:UMDcase} and \ref{cor:HYcase}) are special cases of our main result:
\begin{itemize}
\item If $X$ is a UMD space, there exists an $\alpha\in (0,1/2)$ depending only on $X$ such that $\|T^n\| = O(n^{\alpha})$;
\item If $X = [Y,H]_{\theta}$ (complex interpolation), where $Y$ is a UMD space and $H$ is a Hilbert  space with $\theta\in (0,1)$, then there exists an $\alpha\in (0,(1-\theta)/2)$ depending only on $X$ such that $\|T^n\| = O(n^{\alpha})$.
\end{itemize}
For instance, the above conclusions can be applied to $L^p$-spaces both in the commutative setting and non-commutative setting.
Improvements of \eqref{eq:strongKreisslinear} for Banach function spaces are discussed in Theorems \ref{thm:mainbfs} and \ref{thm61}.

The previously mentioned improvements of \eqref{eq:strongKreisslinear} follow from one single theorem, in which the main condition on $X$ is formulated in terms of upper and lower estimates for decompositions in the Fourier domain, which we introduce and study in detail. The definitions and properties of these decompositions  will be given in Section \ref{fourierdecomposition}.

 \begin{theorem}\label{thm:mainintro}
Let $X$ be a Banach space which has upper $\ell^{q_0}(L^p)$-decomp\-ositions and lower $\ell^{q_1}(L^p)$-decompositions, where $p\in (1, \infty)$ and $1<q_0\leq q_1<\infty$.  If $T$ is a strongly Kreiss bounded operator on $X$, then there exist constants $C,\beta>0$ depending on $X$ and $T$ such that
    \begin{equation*}
        \|T^n\|\leq C n^{\frac12(\frac{1}{q_0}-\frac{1}{q_1})}(\log (n+2))^{\beta}, \ \ n\geq 1.
    \end{equation*}
\end{theorem}
One can see that $q_0=q_1$ would lead to logarithmic growth. However, this equality can only occur if $X$ is isomorphic to a Hilbert space. This follows from Propositions \ref{prop:Utypecotype} and \ref{prop:Ltypecotype} and Kwapien's theorem (see \cite{HNVW2}).

\bigskip

The structure of this paper is as follows. We explain our main tool: Fourier decomposition properties in Section \ref{fourierdecomposition}. With the help of Fourier decompositions, we can prove our main results in Section \ref{UMD} for general $\UMD$ space and in Section \ref{bfs} for $\UMD$ Banach function spaces. In Section \ref{sec:open} we collect some open problems related to the results of paper.

\section{Fourier Decompositions}\label{fourierdecomposition}
In Theorem \ref{thm:mainintro}, we used an assumption in terms of decompositions in the Fourier domain. In this section, we will introduce these concepts.

\subsection{Fourier Multiplier theory}
We start by briefly recalling some Fourier multiplier theory in the vector-valued setting. For details we refer to \cite{HNVW1}.

Let $\T = [0,1]$ denote the torus, and set $e_n(t):=\ee^{2\pi i nt}$ for $t\in \T$ and $n\in \mathbb{Z}$. For $f\in L^1(\mathbb{T};X)$ and $n\in\mathbb{Z}$, let \[\widehat{f}(n):=\mathcal{F}(f)(n):=\int_{\mathbb{T}}f(t)\overline{e_n(t)}\dd t\]
and for a sequence $a = (a_n)_{n \in \Z} \in \ell^1(\Z;X)$ define
$$
\widecheck{a}:=\mathcal{F}^{-1}(a) := \sum_{n \in \Z}a_ne_n.
$$
Let $\cP(\T;X)$ denote the set of all trigonometric $f:\T\to X$.  The space $\cP(\T;X)$ is dense in $L^p(\T;X)$ for all $p\in [1, \infty)$.

For a bounded sequence $m=(m_n)_{n\in\mathbb{Z}} \in \ell^\infty(\Z;\mathcal{L}(X))$ and $f\in \cP(\T;X)$, we define the Fourier multiplier
$$
T_mf := \mathcal{F}^{-1}(m\widehat{f}),
$$
which is well defined since $\widehat{f}$ is a finitely nonzero sequence.
In this definition, we also interpret a bounded function $m \colon \R \to \calL(X)$ as a sequence by setting $m_n:= m(n)$ for $n \in \Z$.

Let $p \in [1,\infty)$. If there exists a constant $C>0$ such that for all $f\in \cP(\T;X)$,
    $$\|T_m f\|_{L^p(\mathbb{T};X)}\leq C\|f\|_{L^p(\mathbb{T};X)},$$
we call $m$ a \emph{$L^p$-Fourier multiplier}. In this case, $T_m$ can be uniquely extended to a bounded linear operator on $L^p(\mathbb{T};X)$.

In order to have a flexible theory of Fourier multipliers, one needs to restrict to the so-called class of {\em $\UMD$ Banach spaces}. Indeed, already the boundedness of the {\em Hilbert transform} $T_m$ on $L^p(\T;X)$ with $p\in (1, \infty)$ and $m = -i\sgn(\cdot)$ gives a characterization of $\UMD$ spaces. Moreover, the same holds for the {\em Riesz projection} $T_m$ where $m = \one_{[0,\infty)}$.
These results are due to Burkholder and Bourgain, and for details we refer to \cite[Chapter 5]{HNVW1}.

We will not give the precise definition of UMD spaces here and one may regard the above characterization as a definition. The following properties of UMD spaces will be used throughout this paper \cite[Chapter 4]{HNVW1}:
\begin{itemize}
\item UMD spaces are (super)reflexive.
\item If $X$ is UMD, then $X^*$ is UMD.
\end{itemize}
Moreover, for $p \in (1,\infty)$, we will write $R_{p,X}:= \norm{T_{\one_{[0,\infty)}}}_{\calL(L^p(\T;X))}.$

\bigskip

We say that a sequence $(a_n)_{n\in \mathbb{Z}}$ of complex numbers is  of \emph{bounded variation} if
\[[a]_{V^1}:=\sum_{n\in \mathbb{Z}}|a_{n+1}-a_n|<\infty,\]
and we denote by $V^1$ the space of all such sequences.
The following vector-valued analogue of the Marcinkiewicz multiplier theorem is due to \cite{Bourgain714}. A corresponding version on $\R$ with  operator-valued $m$ can be found in \cite[Theorem 8.3.9]{HNVW2}. Via transference, the periodic case can also be derived from \cite{HNVW2}.
\begin{lemma}\label{m16}
     Let $X$ be a $\UMD$ space and $1<p<\infty$. If $m\in \ell^\infty \cap V^1$, then $m$ is a Fourier multiplier, and there exists a constant $M_{p,X}\geq 1$ such that
     \[\|T_m\|
     \leq M_{p,X} \bigl(\|m\|_{\ell^\infty}+[m]_{V^1}\bigr).\]
\end{lemma}

\subsection{\texorpdfstring{$\ell^q(L^p)$}{lqLp}-Fourier decompositions}
  After the above preparation we can now introduce the $\ell^q(L^p)$-Fourier decompositions.
For an interval $I\subseteq \Z$ and $f\in \cP(\T;X)$, define
\[D_I f:=  T_{\one_I}f =\mathcal{F}^{-1}(\one_{I}\widehat{f}).\]
A family $\cI$ of subsets of $\Z$ is called an {\em interval partition} if it is a partition of $\Z$ and each $I\in \cI$ is an interval.

\begin{definition}\label{def:decom}
Let $X$ be a Banach space and $p,q\in [1,\infty]$.
\begin{enumerate}[(i)]
\item The space $X$ is said to have \emph{upper $\ell^q(L^p)$-decompositions} if there exists a constant $U>0$ such that for each interval partition $\cI$ and all $f\in \cP(\T;X)$,
  \begin{align*}
 \|f\|_{L^{p}(\mathbb{T};X)}  \leq  U \, \big(\sum_{I\in\cI}\|D_{I}f\|_{L^{p}(\mathbb{T};X)}^{q}\big)^{\frac{1}{q}}.
  \end{align*}
  \item\label{it:decom2} The space $X$ is said to have \emph{lower $\ell^q(L^p)$-decompositions} if there exists a constant $L>0$ such that for each interval partition $\cI$ and all $f\in \cP(\T;X)$,
      \begin{align*}
\big(\sum_{I\in\cI}\|D_{I}f\|_{L^{p}(\mathbb{T};X)}^{q}\big)^{\frac{1}{q}}\leq  L \, \|f\|_{L^{p}(\mathbb{T};X)}.
    \end{align*}
\end{enumerate}
\end{definition}

  By the triangle inequality, it is clear that every Banach space has upper $\ell^1(L^p)$-decompositions  for any $p\in [1,\infty]$. Moreover, if $X$ is nonzero, then $q<\infty$ and in fact $q\leq 2\wedge p'$ for the upper decompositions (see Proposition \ref{prop:Utypecotype}).

 Any $\UMD$ Banach space has lower $\ell^{\infty}(L^p)$-decompositions for  $p\in (1,\infty)$, which follows from the boundedness of the Riesz projection. In Theorem \ref{thm:nontrivialBourgain} we shall see that this can be improved. Moreover, the UMD property and $p\in (1, \infty)$ cannot be avoided for the lower decompositions. Indeed, if there exist $p,q\in [1, \infty]$ such that (a nonzero) $X$ has lower $\ell^q(L^p)$-decompositions, then $p\in (1, \infty)$, and $X$ is a UMD space. To see this, take $I_1 = \Z_+$ and $I_2 = \Z_-\setminus\{0\}$. Definition \ref{def:decom}\ref{it:decom2}  immediately implies that for all $f\in \cP(\T;X)$,
  \[\|D_{I_1} f\|_{L^p(\T;X)} \leq L  \,\|f\|_{L^p(\T;X)},\]
which gives the boundedness of the Riesz projection. Thus, $X$ is a UMD space and $p \in (1,\infty)$.

  \subsection{Basic Properties}
Let us discuss some basic properties of the upper and lower decompositions.
We start with a simple duality result.
\begin{proposition}[Duality]\label{prop:dual}
     Let $X$ be a Banach space, $p\in (1, \infty)$ and $q\in [1, \infty]$. The following are equivalent:
\begin{enumerate}[(1)]
   \item\label{it:dual1} $X$ is a $\UMD$ space which has upper $\ell^q(L^p)$-decompositions.
\item\label{it:dual2} $X^*$ has lower $\ell^{q'}(L^{p'})$-decompositions.
    \end{enumerate}
\end{proposition}
\begin{proof}
We only consider the case $q\in (1, \infty)$.

\ref{it:dual2}$\Rightarrow$\ref{it:dual1}:
We already noted that $X$ is a $\UMD$ space.
To show the upper estimate, let $\cI$ be an interval partition, let $f\in \cP(\T;X)$  and $g\in \cP(\T;X^*)$.
 By H{\"o}lder's inequality and the assumption, we have
\begin{align*}
    \big|\lb f,{g}\rb _{L^{p}(\mathbb{T};X),L^{p'}(\mathbb{T};X^*)}\big|&=\Big|\int_{\mathbb{T}}\lb f,{g}\rb_{X,X^*}\dd t\Big|
    =\Big|\int_{\mathbb{T}}\sum_{I\in \cI}\lb D_{I}f,D_{I}{g}\rb_{X,X^*}\dd t\Big|\\
    &\leq \sum_{I \in \cI}\|D_{I}f\|_{L^p(\mathbb{T};X)}\|D_{I}{g}\|_{L^{p'}(\mathbb{T};X^*)}\\
    &\leq \Big(\sum_{I\in\cI}\|D_{I}f\|_{L^p(\mathbb{T};X)}^{q}\Big)^{\frac{1}{q}}
    \Big(\sum_{I\in\cI}\|D_{I}{g}\|_{L^{p'}(\mathbb{T};X^*)}^{q'}\Big)^{\frac{1}{q'}}\\
    &\leq L\Big(\sum_{I\in\cI}\|D_{I}f\|_{L^{p}(\mathbb{T};X)}^{q}\Big)^{\frac{1}{q}}\|g\|_{L^{p'}(\mathbb{T};X^*)}.
\end{align*}
Taking the supremum over all ${g}$ which satisfy $\|g\|_{L^{p'}(\mathbb{T};X^*)}\leq 1$, it follows from \cite[Proposition 1.3.1]{HNVW1}
that $X$ has upper $\ell^q(L^p)$-decompositions.

\ref{it:dual1}$\Rightarrow$\ref{it:dual2}: Let $g\in \cP(\T;X^*)$. Let $f_I\in \cP(\T;X)$ for $I\in\cI$, where we suppose that only finitely many $f_I$ are nonzero. Let $f:=\sum_{I\in\cI} D_{I}f_I$. H{\"o}lder's inequality and the assumption give that
\begin{align*}
\Big|\sum_{I\in\cI} \int_{\mathbb{T}}\lb f_I,D_{I}{g}\rb_{X,X^*}\dd t\Big|&=\Big|\int_{\mathbb{T}} \lb f,{g}\rb_{X,X^*}\dd t\Big|\leq \|f\|_{L^{p}(\mathbb{T};X)}\|{g}\|_{L^{p'}(\mathbb{T};X^*)}\\
&\leq U \Big(\sum_{I \in \cI}\|D_{I}f\|_{L^{p}(\mathbb{T};X)}^{q}\Big)^{\frac{1}{q}}\|{g}\|_{L^{p'}(\mathbb{T};X^*)}\\
&\leq 2 U R_{p,X} \Big(\sum_{I \in \cI}\|f_I\|_{L^{p}(\mathbb{T};X)}^{q}\Big)^{\frac{1}{q}}\|{g}\|_{L^{p'}(\mathbb{T};X^*)}.
\end{align*}
where, in the last step, we applied the boundedness of the Riesz projection.
Taking the supremum over all $(f_I)_{I \in \cI}$ such that $\sum_{I\in \cI}\|f_I\|_{L^{p}(\mathbb{T};X)}^{q}\leq 1$, it follows from \cite[Proposition 1.3.1]{HNVW1} that $X^*$ has lower $\ell^{q'}(L^{p'})$-decompositions.
\end{proof}

In the following proposition, we show that one can trade $\ell^q$-summability for polynomial growth in the number of intervals in the decomposition properties,
which seems like a natural way to prove upper decompositions. A similar result holds for the lower decompositions case.
\begin{proposition}[$\ell^q$-summability versus growth $\alpha$]\label{prop:growthtrick}
Let $X$ be a Banach space and let $p,q \in [1,\infty]$.
\begin{enumerate}[(1)]
\item \label{it:alpha1} If $X$ has upper $\ell^{q}(L^{p})$-decompositions, then
there exists a constant $U>0$ such that for  $r\in [q, \infty]$, all finite families of disjoint intervals $\cI$ and  $f\in \cP(\T;X)$ with support in $\cup \{I\in \cI\}$,
\begin{align*}
 \| f\|_{L^{p}(\mathbb{T};X)}  \leq  U \,(\#\cI)^{\frac{1}{q}-\frac{1}{r}} \Big(\sum_{I\in\cI}\|D_{I}f\|_{L^{p}(\mathbb{T};X)}^{r}\Big)^{\frac{1}{r}}.
  \end{align*}

\item \label{it:alpha2} Conversely, if there exists an $r\in (q, \infty]$ and
a constant $U>0$ such that for all finite families of disjoint intervals $\cI$ and  $f\in \cP(\T;X)$ with support in $\cup \{I\in \cI\}$,
\begin{align*}
 \|f\|_{L^{p}(\mathbb{T};X)}  \leq  U \,(\#\cI)^{\frac1q-\frac{1}{r}} \Big(\sum_{I\in\cI}\|D_{I}f\|_{L^{p}(\mathbb{T};X)}^{r}\Big)^{\frac{1}{r}},
  \end{align*}
then $X$ has upper $\ell^{s}(L^{p})$-decompositions for  $1\leq s<q$.
\end{enumerate}
\end{proposition}

\begin{proof}
Note that \ref{it:alpha1} follows directly from H\"{o}lder's inequality. For \ref{it:alpha2},  let $\cI = \{I_{k}:k\geq 1\}$ be an interval partition of $\Z$. For a trigonometric polynomial $f:\T\to X$, set $f_k:=D_{I_k}f$. Without loss of generality, we may assume that $$\|f_k\|_{L^{p}(\mathbb{T};X)}\ge \|f_{k+1}\|_{L^{p}(\mathbb{T};X)}, \qquad k \geq 1.$$ By the triangle inequality and the assumption, we get
    \begin{align*}
       \| f\|_{L^{p}(\mathbb{T};X)}&= \Big\|\sum_{k=1}^{\infty}f_k\Big\|_{L^{p}(\mathbb{T};X)}
       \leq \sum_{j=1}^\infty\Big\|\sum^{2^j-1}_{k=2^{j-1}}f_k\Big\|_{L^{p}(\mathbb{T};X)}\\
       & \leq  U \sum_{j=1}^\infty 2^{(j-1)(\frac1q-\frac{1}{r}) }\Big(\sum^{2^j-1}_{k=2^{j-1}}\|f_k\|_{L^{p}(\mathbb{T};X)}^r\Big)^{\frac{1}{r}}\\
        &\leq  U\sum_{j=1}^\infty2^{(j-1)\frac1q}\|f_{2^{j-1}}\|_{L^{p}(\mathbb{T};X)}\\
        &\leq U\sum_{j=1}^\infty2^{(j-1)(\frac{1}{q}-\frac{1}{s})}\cdot \Big(\sum_{k=1}^\infty\|f_k\|^{s}_{L^p(\T;X)}\Big)^{\frac{1}{s}}.
    \end{align*}
Since $\frac{1}{q}-\frac{1}{s}<0$, assertion \ref{it:alpha2}  follows.
\end{proof}

In the next proposition, we discuss a complex interpolation result for the decomposition properties.
\begin{proposition}[Interpolation]\label{prop:interpolation}
Let $(X_0, X_1)$ be an interpolation couple of  $\UMD$ spaces. Let $p_0, p_1 \in (1,\infty)$ and $q_0, q_1\in [1, \infty]$. Let $\theta\in (0,1)$, set $X_{\theta} = [X_0, X_1]_{\theta}$ and
\[\frac{1}{p} = \frac{1-\theta}{p_0} + \frac{\theta}{p_1}, \qquad \frac{1}{q} = \frac{1-\theta}{q_0} + \frac{\theta}{q_1}.\]
If $X_i$ has upper (lower) $\ell^{q_i}(L^{p_i})$-decompositions for $i=0,1$, then $X_{\theta}$ has upper (lower) $\ell^{q}(L^{p})$-decompositions.
\end{proposition}
\begin{proof}
We start with the proof of the lower case. Since $X_{\theta}$ is a UMD space, it has lower $\ell^\infty(L^p)$-decompositions. Thus we may assume without loss of generality that $q<\infty$ and thus $\min\{q_0, q_1\}<\infty$. Let $\cI$ be an interval partition of $\Z$. Let $$T:L^{p_i}(\T;X_i)\to \ell^{q_i}(\cI;L^{p_i}(\T;X_i))$$ be given by $T f = (D_{I}f)_{I\in \cI}$ for $i= 0,1$. From the assumption we see that $T$ is bounded of norm $L_i$ for $i=0,1$. Therefore, by complex interpolation (see \cite[Theorem 2.2.6]{HNVW1}), we obtain that $T:L^{p}(\T;X_\theta)\to \ell^{q}(L^{p}(\T;X_\theta))$ is bounded and
\[\|T\|_{\mathcal{L}(L^{p}(\T;X_\theta), \ell^{q}(L^{p}(\T;X_\theta)))}\leq L_0^{1-\theta} L_1^{\theta}.\]
This gives the required result.

For the upper case, in Proposition \ref{prop:Utypecotype}, we will show that $q_0,q_1<\infty$. Let $T \colon \ell^{q_i}(\cI;L^{p_i}(\T;X_i)) \to L^{p_i}(\T;X_i)$ be given by
$$
T((f_{I})_{I \in \cI}):=  \sum_{I \in \cI} D_I f_I
$$
for $i=0,1$. Note that
\begin{align*}
  \norm{T((f_{I})_{I \in \cI})}_{L^{p_i}(\T;X_i)} &= \nrms{\sum_{I \in \cI} D_I f_I}_{L^{p_i}(\T;X_i)}\\
  &\leq U_i \Bigl(\sum_{J \in \cI}\nrms{D_J \sum_{I \in \cI} D_I f_I}_{L^{p_i}(\T;X_i)}^{q_i}\Bigr)^{\frac1{q_i}}\\
  &\leq 2 U_i R_{p_i,X_i}\norm{(f_{I})_{I \in \cI}}_{\ell^{q_i}(\cI;L^{p_i}(\T;X_i))}.
\end{align*}
Therefore, again by complex interpolation (see \cite[Theorem 2.2.6]{HNVW1}), we obtain that $T \colon \ell^{q}(\cI;L^{p}(\T;X_{\theta})) \to L^{p}(\T;X_{\theta})$ is bounded. Applying this to $(f_{I})_{I \in \cI} = (D_If)_{I \in \cI}$ for $f \in \cP(\T;X_{\theta})$ yields the result.
\end{proof}

With a similar method we obtain the following ``extrapolation result''.
\begin{proposition}[Extrapolation]\label{prop:extrapolation-p}
    Let $X$ be a $\UMD$ space, $p\in (1, \infty)$, and $q\in (1, \infty)$.
\begin{enumerate}[(1)]
\item\label{it:extrap1}  If $X$ has upper $\ell^q(L^p)$-decompositions, then $X$ has upper $\ell^{s}(L^r)$-decomp\-ositions for every $s\in[1,q)$  and $r \in (1,\infty)$ such that
\[s' = \frac{q'}{\theta}\quad  \text{for} \quad \theta < \min\Big\{\frac{p}r, \frac{p'}{r'}\Big\}.\]
\item\label{it:extrap2} If $X$ has lower $\ell^q(L^p)$-decompositions, then $X$ has lower $\ell^{s}(L^r)$-decomp\-ositions for every $s\in(q, \infty]$ and $r \in (1,\infty)$ such that
\[s = \frac{q}{\theta} \quad \text{for} \quad \theta < \min\Big\{\frac{p}r, \frac{p'}{r'}\Big\}.\]
\end{enumerate}
\end{proposition}
\begin{proof}
By the duality result in Proposition \ref{prop:dual}, it suffices to prove \ref{it:extrap2}. Moreover, $X$ has lower $\ell^\infty(L^t)$-estimates for all $t\in (1, \infty)$ by the boundedness of the Riesz projection. It therefore follows from Proposition \ref{prop:interpolation} that $X$ has lower $\ell^s(L^r)$-estimates if $\frac{1}{s} = \frac{1-\theta}{\infty} + \frac{\theta}{q} = \frac{\theta}{q}$ and $\frac{1}{r} = \frac{1-\theta}{t} + \frac{\theta}{p}$. First consider $r>p$. Since we assumed $\frac{\theta}{p}<\frac{1}{r}$, the latter identity holds for some $t \in (r,\infty)$. If $r<p$, then using  $\frac{\theta}{p'}<\frac{1}{r'}$, one can check that this identity holds for some $t\in (1,r)$.
\end{proof}

The decomposition properties also behave well in the following sense, where we note that
extrapolation to other exponents can be deduced from Proposition \ref{prop:growthtrick} and Corollary \ref{prop:extrapolation-p}.
 \begin{proposition}\label{prop:LparoundX}
Let $(S, \mathcal{A}, \mu)$ be a $\sigma$-finite measure space. Let $X$ be a Banach space and let $p,q\in (1,\infty)$.
\begin{enumerate}[(1)]
\item\label{it:Lpupper} If $X$ has upper $\ell^q(L^p)$-decompositions, then $L^p(S;X)$ has upper $\ell^{p\wedge q}(L^p)$-decompositions.
\item\label{Lplower} If $X$ has lower $\ell^q(L^p)$-decompositions, then $L^p(S;X)$ has lower $\ell^{p\vee q}(L^p)$-decompositions.
\end{enumerate}
\end{proposition}
\begin{proof}
\ref{it:Lpupper}: By Fubini's theorem, the assumption, the contractive embedding $\ell^{p\wedge q}\hookrightarrow \ell^q$, and Minkowski's inequality, we obtain
\begin{align*}
 \|f\|_{L^p(\mathbb{T};L^p(S;X))} & = \|f\|_{L^p(S;L^p(\mathbb{T};X))}
 \\ & \leq  U \|(D_{I} f)_{I\in\cI}\|_{L^p(S;\ell^q(\cI;L^p(\mathbb{T};X)))}
 \\ & \leq U \|(D_{I} f)_{I\in\cI}\|_{L^p(S;\ell^{p\wedge q}(\cI;L^p(\mathbb{T};X)))}
 \\ & \leq U\|(D_{I} f)_{I\in\cI}\|_{\ell^{p\wedge q}(\cI;L^p(S;L^p(\mathbb{T};X)))}
 \\ & =U\|(D_{I} f)_{I\in\cI}\|_{\ell^{p\wedge q}(\cI;L^p(\mathbb{T};L^p(S;X)))}.
    \end{align*}

\ref{Lplower}: This can be proved in the same way.
\end{proof}

The following result is much deeper and follows from \cite{bourgain} and  \cite{fabian}. It will play a role in some of the results below.
\begin{theorem}\label{thm:nontrivialBourgain}
Let $X$ be a Banach space and $p\in (1, \infty)$.
\begin{enumerate}[(1)]
\item\label{it:bourgain1}
 $X$ is super-reflexive if and only if there exists a $q\in (1,\infty)$ such that $X$ has upper $\ell^q(L^p)$-decompositions.
\item\label{it:bourgain2} $X$ is a $\UMD$ space if and only if there exists a $q\in (1,\infty)$ such that $X$ has lower $\ell^q(L^p)$-decompositions.
\end{enumerate}
\end{theorem}
\begin{proof}
\ref{it:bourgain1} is immediate from \cite[Theorem 10]{bourgain} and  \cite[Theorem 9.25]{fabian}.
For \ref{it:bourgain2}, note that if $X$ is a $\UMD$ space, then $X^*$ is a $\UMD$ space as well, and thus super-reflexive. By \ref{it:bourgain1}  for each $p'\in (1,\infty)$, there exists a $q'\in (1,\infty)$ such that $X^*$ has upper $\ell^{q'}(L^{p'})$-decompositions. Applying Proposition \ref{prop:dual}, $X$  has lower $\ell^{q}(L^{p})$-decompositions. The converse implication has already been observed below  Definition \ref{def:decom}.
\end{proof}

\subsection{Necessity of type and cotype properties}\label{subsec:necctype}
We have already seen that super-reflexivity and UMD are necessary for upper and lower decompositions, respectively.
Our next aim is to show that the decomposition properties also imply (Fourier) type and cotype. We briefly recall the definitions. For  details the reader is referred to \cite{HNVW1, HNVW2, HNVW3}.

Let $p\in [1, \infty]$. The space $X$ has {\em Fourier type $p$} if there exists a constant $\varphi_{p,X}>0$ such that for all finitely nonzero $(x_n)_{n\in \Z}$  in $X$, we have
\[\Big\|\sum_{n\in \Z} e_n x_n\Big\|_{L^{p'}(\T;X)}\leq \varphi_{p,X}\,\|(x_n)_{n\geq 1}\|_{\ell^p(\Z;X)}.\]
Every Banach space has Fourier type $1$. Moreover, from the scalar case it follows that necessarily $p\in [1, 2]$. Finally, note that $X$ has Fourier type $p$ if and only if $X^*$ has Fourier type $p$ (see \cite[Propositions 2.4.16 and 2.4.20]{HNVW1}).

Let $(\varepsilon_n)_{n\geq 1}$ be a complex Rademacher sequence on a probability space $(\O,\mathcal{A},\P)$, i.e., \ a sequence of independent random variables $\varepsilon_n$ which are uniformly distributed on the unit circle in $\C$.
Let $p\in [1, \infty)$. The space $X$ is said to have \emph{type $p$} if there exists a constant $\tau_{p,X} > 0$ such that for all $x_1,\cdots, x_n \in  X$, we have
  $$\Big\|\sum_{k=1}^n\varepsilon_k x_k\Big\|_{L^2(\Omega;X)}\leq \tau_{p,X} \has{\sum_{k=1}^n \|x_k\|_{X}^p}^{\frac1p}.$$

Let $q\in [1, \infty]$. The space $X$ is said to have \emph{cotype $q$} if there exists a constant $c_{q,X}> 0$ such that for all $x_1,\cdots, x_n \in  X$, we have
  $$ \has{\sum_{k=1}^n \|x_k\|_{X}^q}^{\frac1q}\leq c_{q,X}\Big\|\sum_{k=1}^n\varepsilon_n x_n\Big\|_{L^2(\Omega;X)}.$$

In the above, the complex Rademacher sequence can be replaced by a real Rademacher sequence (see \cite[Proposition 6.1.19]{HNVW2}).
Every space has type $1$ and cotype $\infty$. Moreover, considering the scalar case one sees that necessarily $p\in [1, 2]$ and $q\in [2, \infty]$ in the above definitions. If $X$ has type $p$, then $X^*$ has cotype $p'$. The converse holds if $X$ has some nontrivial type, which for instance is the case if $X$ is UMD or super-reflexive. Finally, note that Fourier type $p$ implies type $p$ and cotype $p'$ (see \cite[Proposition 7.3.6]{HNVW2}).

To deduce type and cotype properties, we will present the details in the case of upper decompositions. The lower case will be derived by duality.

\begin{proposition}[Upper decompositions implies type and cotype]\label{prop:Utypecotype}
Let $X$ be a Banach space and $p,q\in [1, \infty]$. If $X$ has upper $\ell^q(L^p)$-decompositions, then $q\in [1, p'\wedge 2]$ and
\begin{enumerate}[(1)]
\item\label{it:type1} $X$ has type $q$;
\item\label{it:type2}  $X$ has Fourier type $r'$ and cotype $r$ for any $r\in (\frac{2q'}{p\wedge 2},\infty)$.
\end{enumerate}
\end{proposition}
\begin{proof}
By the assumption applied to the trigonometric polynomial $f = \sum_{k=1}^n e_k x_k$, and $I_k = \{k\}$ for $k\in \Z$, we obtain
\begin{align}\label{eq:Fouriertypeest}
\Big\|\sum_{n\in \Z} e_n x_n\Big\|_{L^{p}(\T;X)}\leq U\, \|(x_n)_{n\geq 1}\|_{\ell^q(\Z;X)}.
\end{align}
This implies $q\in [1,p']$. Indeed, if $q>p'$, this would lead to an improvement of the classical Hausdorff--Young inequalities, which is known to be false for $\C$ and thus for one-dimensional subspaces of $X$. This can for instance be deduced from \cite[below (4.6) with $a\in (0,1)$]{DomVer}.

\ref{it:type1}: The fact that $X$ has type $q$, and thus in particular $q\leq 2$, follows from  \eqref{eq:Fouriertypeest} and the same argument as in \cite[Proposition 7.3.6]{HNVW2}.

\ref{it:type2}: By H\"older's inequality we may assume $p \leq 2$ in \eqref{eq:Fouriertypeest}.  Interpolating this estimate with the trivial bound
\begin{align*}
\Big\|\sum_{k=1}^n e_k x_k\Big\|_{L^\infty(\T;X)}\leq \sum_{k=1}^n \|x_k\|,
\end{align*}
and setting $\theta = p/2$ and $\frac1s = \frac{1-\theta}{1} + \frac{\theta}{q} = 1- \frac{p}{2 q'}$, we obtain
\begin{align*}
\Big\|\sum_{k=1}^n e_k x_k\Big\|_{L^2(\T;X)}\leq U^{p/2} \Big(\sum_{k=1}^n \|x_k\|^s\Big)^{1/s}.
\end{align*}
Note that $s\in [1, 2]$ as a consequence of $q\leq p'$. Therefore, H\"older's inequality implies that
\begin{align*}
\Big\|\sum_{k=1}^n e_k x_k\Big\|_{L^2(\T;X)} \leq U^{p/2} n^{\frac{1}{s}-\frac12} \Big(\sum_{k=1}^n \|x_k\|^2\Big)^{1/2}.
\end{align*}
By \cite[Lemma 13.1.32]{HNVW3} the latter estimate implies Fourier type $r'$ if $\frac{1}{r'}>\frac{1}{s} = 1- \frac{p}{2 q'}$, which is the required result. Since Fourier type $r'$ implies cotype $r$ by \cite[Proposition 7.3.6]{HNVW2}, this completes the proof.
\end{proof}

\begin{proposition}[Lower decompositions implies type and cotype]\label{prop:Ltypecotype}
Let $X$ be a Banach space, $p\in (1, \infty)$ and $q\in [1, \infty]$. If $X$ has lower $\ell^q(L^p)$-decompositions, then $q\in [p'\vee 2,\infty]$ and
\begin{enumerate}[(1)]
\item\label{it:cotype1}  $X$ has cotype $q$;
\item\label{it:cotype2}  $X$ has Fourier type $r'$ and type $r'$ for any $r\in (\frac{2q}{p'\wedge 2},\infty)$.
\end{enumerate}
\end{proposition}
\begin{proof}
Note that the assumption implies that $X$ is UMD. By Proposition \ref{prop:dual}, we know that $X^*$ has upper $\ell^{q'}(L^{p'})$-decompositions. Thus Proposition \ref{prop:Utypecotype} gives that $q'\in [1, p\wedge 2]$, and $X^*$ has type $q'$. Therefore, $X$ has cotype $q$ and this proves \ref{it:cotype1}. Similarly, $X^*$ has Fourier type $r'$ for any $r\in (\frac{2q}{p'\wedge 2},\infty)$. This implies that $X$ has Fourier type $r'$, and thus also type $r'$ by \cite[Proposition 7.3.6]{HNVW2}.
\end{proof}

\subsection{Examples}

We have already seen that every UMD space admits nontrivial upper and lower $\ell^q(L^p)$-decompositions. In this section, we give some concrete spaces and indicate what the admissible $p$ and $q$ are on these spaces.

\begin{example}\label{ex:HilbertUL}
Let $X$ be a Hilbert space. Then
\begin{itemize}
\item $X$ has upper $\ell^{p}(L^p)$-decompositions for  $p\in (1, 2]$.
\item $X$ has lower $\ell^{p}(L^p)$-decompositions for  $p\in [2, \infty)$.
\end{itemize}
Indeed, the first claim follows from the second by the duality statement in Proposition \ref{prop:dual}. Moreover, for the second statement it suffices to consider $X=\C$ by \cite[Theorem 2.1.9]{HNVW1}. By Rubio de Francia's Littlewood--Paley inequality for arbitrary intervals \cite{Rubio}, which for $\T$ can be found in \cite{KP05}, there is a $C>0$ such that  for $p \in [2,\infty)$, each interval partition $\cI$ and all $f\in \cP(\T;X)$, we have
\begin{align}\label{eq:Rubiofun}
\has{\sum_{I\in\cI} \norm{D_{I}f}_{L^{p}(\mathbb{T})}^p}^{\frac1p} \leq \nrms{\has{\sum_{I\in\cI}|D_{I}f|^2}^{\frac12}}_{L^{p}(\mathbb{T})}\leq  C \, \|f\|_{L^{p}(\mathbb{T})}.
\end{align}

For the other ranges of $p$, we can combine \eqref{eq:Rubiofun} for $p=2$ with Proposition \ref{prop:extrapolation-p} to obtain that
\begin{itemize}
\item $X$ has upper $\ell^{q}(L^p)$-decompositions for  $p\in [2, \infty)$ and $q \in [1,p')$.
\item $X$ has lower $\ell^{q}(L^p)$-decompositions for  $p\in (1, 2]$ and $q \in (p',\infty]$.
\end{itemize}
The endpoint $q=p'$ is missing in the above result due to the use of Proposition \ref{prop:extrapolation-p}. We leave it as an open problem whether these endpoints hold, see Problem \ref{prob:scalarfield}.
\end{example}

\begin{example}\label{ex:LpUL}
Let $(S, \mathcal{A},\mu)$ be a $\sigma$-finite measure space and $p\in (1, \infty)$. From Example \ref{ex:HilbertUL}  and Proposition \ref{prop:LparoundX} we immediately obtain
\begin{itemize}
\item $L^p(S)$ has upper $\ell^{p}(L^p)$-decompositions for $p \in (1,2]$ and upper  $\ell^{q}(L^p)$-decompositions
 for  $p\in [2, \infty)$ and $q \in [1,p')$.
\item $L^p(S)$ has lower $\ell^{p}(L^p)$-decompositions for $p \in [2,\infty)$ and lower $\ell^{q}(L^p)$-decompositions for  $p\in (1, 2]$ and $q \in (p',\infty]$.
\end{itemize}
The claims about $\ell^p(L^p)$-decompositions are optimal, which follows from the optimality of Corollary \ref{cor:Lpcase} below. Whether the endpoints $q =  p'$ hold is even unclear in the case $S$ is a singleton thus $L^p(S) = \C$, see Problem \ref{prob:scalarfield}.

The spaces $X = L^1(S)$ and $X = L^\infty(S)$ do not have nontrivial upper and lower estimates, since they are not reflexive in general.
\end{example}

An efficient method to create many examples can be obtained by interpolation. It is actually an open problem if all UMD spaces can be written as an interpolation space as below. For UMD lattices this is indeed the case (see \cite{Rubio2}). Moreover, noncommutative $L^p$-spaces can also be written in the form below.
\begin{example}\label{710}
Let $X:=[Y,H]_{\theta}$, where $Y$ is a $\UMD$ Banach space and $H$ is a Hilbert space such that $(Y,H)$ is an interpolation couple, and $\theta\in (0,1)$. Let $p\in ((1-\frac{\theta}{2})^{-1}, \frac{2}{\theta})$.
Then there exists a $\theta_0>\theta$ depending on $\theta,p$ and $Y$ such that
\begin{itemize}
\item $X$ has lower $\ell^{\frac{2}{\theta_0}}(L^p)$-decompositions.
\item $X$ has upper $\ell^{\frac{2}{2-\theta_0}}(L^{p})$-decompositions.
\end{itemize}
As a trivial consequence, the same holds with $\theta_0=\theta$. To derive the above, we only explain the lower case as the upper case can be proved similarly.
By the assumption on $p$ we can find $p_0\in (1, \infty)$ such that
\[\frac{1}{p} = \frac{1-\theta}{p_0}+\frac{\theta}{2}.\]
By Theorem \ref{thm:nontrivialBourgain} there exists an $s\in (1, \infty)$ such that $Y$ has lower $\ell^s(L^{p_0})$-decompositions.
Note that $s\geq p'_0\vee 2$ by Proposition \ref{prop:Ltypecotype}. Since $H$ has lower $\ell^2(L^2)$-decompositions, Proposition \ref{prop:interpolation} gives that $X$ has lower $\ell^r(L^p)$-decompositions where $r\in [2, s]$ satisfies $\frac{1}{r} = \frac{1-\theta}{s} + \frac{\theta}{2}$. This gives the result in the lower case.
 \end{example}

\section{Main Results on \texorpdfstring{$\UMD$}{UMD} Banach Spaces}\label{UMD}

\subsection{Statement of the results}
In this section, we prove Theorem \ref{thm:mainintro} and discuss several consequences. We use a slightly more general formulation below, as this is required to obtain sharp estimates in Corollary \ref{cor:Lpcase}. The main extra ingredient is to allow growth in the upper and lower decompositions.

\begin{theorem}\label{thm:main}
Let $X$ be a Banach space, let $p, q_0,q_1\in (1, \infty)$ and let $\gamma_0\in [0,1/q_0')$, $\gamma_1\in [0,1/q_1)$. Suppose that the following conditions hold:
\begin{enumerate}[(1)]
\item\label{it:main1} There exists a constant $U>0$ such that for all finite families of disjoint intervals $\cI$ and all $f\in \cP(\T;X)$ with support in $\cup \{I\in \cI\}$,
  \begin{align*}
 \|f\|_{L^{p}(\mathbb{T};X)}  \leq  U \, (\#\cI)^{\gamma_0} \big(\sum_{I\in\cI}\|D_{I}f\|_{L^{p}(\mathbb{T};X)}^{q_0}\big)^{\frac{1}{q_0}};
  \end{align*}
\item\label{it:main2} There exists a constant $L>0$ such that for all finite families of disjoint intervals $\cI$ and all $f\in \cP(\T;X)$ with support in $\cup \{I\in \cI\}$,
\begin{align*}
\big(\sum_{I\in\cI}\|D_{I}f\|_{L^{p}(\mathbb{T};X)}^{q_1}\big)^{\frac{1}{q_1}}\leq  L\, (\#\cI)^{\gamma_1} \, \|f\|_{L^{p}(\mathbb{T};X)}.
\end{align*}
\end{enumerate}
Suppose that $T\in \calL(X)$ is strongly Kreiss bounded with constant $K_s$. Then there exist constants $C,\beta>0$ depending on $X$ and $K_s$ such that
    \begin{equation*}
        \|T^n\|\leq Cn^{\frac12(\frac{1}{q_0}-\frac{1}{q_1}+\gamma_0+\gamma_1)}(\log (n+2))^{\beta}, \qquad n\geq 1.
    \end{equation*}
\end{theorem}

For $\gamma_0= \gamma_1=0$, the conditions \ref{it:main1} and \ref{it:main2} in Theorem \ref{thm:main} are equivalent to the upper $\ell^{q_0}(L^p)$-decompositions and lower $\ell^{q_1}(L^p)$-decompositions of $X$, respectively. In many cases it is sufficient to consider $\gamma_0 = \gamma_1 = 0$. Moreover, note that  by Proposition \ref{prop:growthtrick}, the estimate in \ref{it:main1} implies that $X$ has upper $\ell^s(L^p)$-decompositions for all $s$ satisfying $\frac1s>\frac{1}{q_0}+\gamma_0$. In particular, this shows that $\frac{1}{q_0}+\gamma_0\geq \frac12$ (see Proposition \ref{prop:Utypecotype}). A similar implication holds from \ref{it:main2} to lower decompositions of $X$, and one has $\frac{1}{q_1'}+\gamma_1\geq \frac12$. Finally, note that it is not useful to consider $\gamma_0\geq 1/q_0'$ or $\gamma_1\geq 1/q_1$, because the obtained bound in the theorem would be worse than \eqref{eq:strongKreisslinear}.

Before we turn to the proof, we derive several immediate consequences. Using Theorem \ref{thm:nontrivialBourgain}, we obtain:
\begin{corollary}[General $\UMD$ case]\label{cor:UMDcase}
Let $X$ be a $\UMD$ Banach space. Suppose that $T\in \calL(X)$ is strongly Kreiss bounded with constant $K_s$. Then there exists an $\alpha\in [0,\frac12)$ depending on $X$, and a constant $C$ depending on $X$ and $K_s$ such that
\begin{equation*}
\|T^n\|\leq C n^{\alpha}, \qquad n\geq 1.
\end{equation*}
\end{corollary}

For interpolation spaces {we can also provide explicit growth rates.}
\begin{corollary}[Intermediate UMD]\label{cor:HYcase}
Let $X:=[Y,H]_{\theta}$, where $Y$ is a $\UMD$ Banach space and $H$ is a Hilbert space such that $(Y,H)$ is an interpolation couple, and $\theta\in (0,1)$. Suppose that $T\in \calL(X)$ is strongly Kreiss bounded with constant $K_s$.
Then there exists an $\alpha \in [0,\frac{1-\theta}{2})$ depending on $X$, and a constant $C>0$ depending on $X$ and $K_s$ such that
\begin{equation*}
\|T^n\| = C n^\alpha (\log(n+2))^{\beta} , \qquad  n\geq 1.
\end{equation*}
\end{corollary}
In particular, one can also take $\alpha = (1-\theta)/2$ in the above.
\begin{proof}
By Example \ref{710} we know that $X$ has upper $\ell^{\frac{2}{2-\theta_0}}(L^2)$-decompositions for some $\theta_0>\theta$, and lower $\ell^{\frac{2}{\theta_{0}}}(L^2)$-decompositions for some $\theta_0>\theta$. Thus it remains to observe that
$\alpha:=\frac{2-\theta_0}{4}-\frac{\theta_{0}}{4} = \frac{1-\theta_0}{2}<\frac{1-\theta}{2}$.
\end{proof}

Similarly, the results of \cite[Theorem 4.5]{CCEL} follow from Example \ref{ex:HilbertUL}.
\begin{corollary}[Hilbert spaces]\label{cor:HScase}
Let $X$ be a Hilbert space. Suppose that $T\in \calL(X)$ is strongly Kreiss bounded with constant $K_s$. Then there exist constants $C,\beta>0$ depending $K_s$ such that
\begin{equation*}
\|T^n\| = C (\log(n+2))^\beta , \qquad n\geq 1.
\end{equation*}
\end{corollary}

We can also recover \cite[Theorem 1.1]{AC}, for which we will need the parameters $\gamma_0,\gamma_1$ in Theorem \ref{thm:main}. Recall that in {\cite[Proposition 1.2]{AC},} it is also shown that the exponent $|\frac{1}{2}-\frac{1}{p}|$ cannot be improved.
\begin{corollary}[$L^p$-spaces]\label{cor:Lpcase}
Let $(S, \mathcal{A}, \mu)$ be a $\sigma$-finite measure space and let $X=L^p(S)$ with $p\in (1, \infty)$. Suppose that $T\in \calL(X)$ is strongly Kreiss bounded with constant $K_s$. Then there exist constants $C,\beta>0$ depending on $p$ and $K_s$ such that
\begin{equation*}
\|T^n\| = Cn^{|\frac{1}{2}-\frac{1}{p}|} (\log(n+2))^\beta , \qquad n\geq 1.
\end{equation*}
\end{corollary}
\begin{proof}
Due to the missing endpoint, using Example \ref{ex:LpUL} would yield the asymptotic $n^\alpha$ for $\alpha > |\frac{1}{2}-\frac{1}{p}|$. We therefore argue differently, using the growth parameters $\gamma_0,\gamma_1$ in Theorem \ref{thm:main}.

By duality, it suffices to consider the case $p \in (1,2]$. By Example \ref{ex:LpUL} we know that assumption \ref{it:main1} in Theorem \ref{thm:main} holds with $q_0 = p$ and $\gamma_0=0$.
Next we claim that assumption \ref{it:main2} in Theorem \ref{thm:main} is satisfied with $q_1=2$ and $\gamma_1 = \frac1p-\frac12$. This readily follows from \cite{AC}. For convenience we include the details.
Note that for finite families of disjoint intervals $\cI$ and all $f\in \cP(\T)$ with support in $\cup \{I\in \cI\}$,
\begin{align*}
\nrms{\has{\sum_{I\in\cI}|D_{I}f|^2}^{\frac12}}_{L^{2}(\mathbb{T})}&= \|f\|_{L^{2}(\mathbb{T})},\\
\nrms{\has{\sum_{I\in\cI}|D_{I}f|^2}^{\frac12}}_{L^{1,\infty}(\mathbb{T})}&\leq  C\, (\#\cI)^{\frac12} \, \|f\|_{L^{1}(\mathbb{T})}.
\end{align*}
The first identity follows from the fact that $\mathcal{F}:L^2(\T)\to \ell^2(\Z)$ is an invertible isometry (see \cite[Theorem 2.1.9]{HNVW1}) and the second one is immediate from the boundedness of the Riesz projection from $L^1(\ell^2)$ into $L^{1,\infty}(\ell^2)$ (see \cite{KP05}). Applying the Marcinkiewicz interpolation theorem (see \cite[Theorem 2.2.3]{HNVW1}) yields
\begin{align*}
 \has{\sum_{I\in\cI}\norm{D_{I}f}_{L^{p}(\mathbb{T})}^2}^{\frac12} &\leq  \nrms{\has{\sum_{I\in\cI}|D_{I}f|^2}^{\frac12}}_{L^{p}(\mathbb{T})} \leq c_p \, C^{\frac2p-1}\, (\#\cI)^{\frac1p-\frac12} \|f\|_{L^{p}(\mathbb{T})}.
\end{align*}
By Minkowski's inequality and Fubini's theorem, we obtain
\begin{align*}
 \has{\sum_{I\in\cI}\norm{D_{I}f}_{L^{p}(\mathbb{T};L^p(S))}^2}^{\frac12} &\leq c_p \, C^{\frac2p-1}\, (\#\cI)^{\frac1p-\frac12} \|f\|_{L^{p}(\mathbb{T};L^p(S))},
\end{align*}
which implies the claim.

From the above and Theorem \ref{thm:main} we see that
    \begin{equation*}
        \|T^n\|\leq Cn^{\alpha}(\log (n+2))^{\beta}, \qquad n\geq 1,
    \end{equation*}
    with
    $$
    \alpha = \frac12\has{\frac{1}{q_0}-\frac{1}{q_1}+\gamma_0+\gamma_1} =\frac12\has{\frac{1}{p}-\frac{1}{2}+0+\frac1p-\frac12} =  \frac{1}{p} - \frac{1}{2},
    $$
    finishing the proof.
\end{proof}

A further application for Banach function spaces will be presented in Theorems \ref{thm:mainbfs} and \ref{thm61}.

\subsection{Preparatory lemmas}
Before we prove Theorem \ref{thm:main}, we need several preparatory lemmas. We start by noting the key property that we will use of strongly Kreiss bounded operators, which follows from  \cite{GZ} and \cite[Corollary 3.2]{MSZ}.
\begin{lemma}\label{lemma:keyKreiss}
     If $T$ is a strongly Kreiss bounded operator on a Banach space $X$ with constant $K_s$, then we have
     \begin{align}\label{6.3}
      \Big\|\sum^n_{k=0} \lambda^kT^k \Big\|\leq  20K_s\, (n+1),\qquad |\lambda|=1,\, n\in \N .
  \end{align}
\end{lemma}
\begin{proof}
It was shown in \cite{GZ} that if $T$ is a strongly Kreiss bounded operator with constant $K_s$, then we have
    \begin{align*}
        \underset{n\ge 0}{\sup}\Big\|\sum^{n}_{k=0}\frac{T^k}{\lambda^{k+1}}\Big\|&\leq \frac{4K_s}{|\lambda|-1},&&|\lambda|>1.\end{align*}
    By \cite[Corollary 3.2]{MSZ}, this is equivalent to
  \begin{align*}
  \Big\|\sum^{n}_{k=0}{\lambda^{k}{T^k}}\Big\|&\leq 20K_s\,(n+1),&&|\lambda|=1,\, n\geq 5.
  \end{align*}
 { If $n\leq 4$, (\ref{6.3}) holds because of (\ref{eq:strongKreisslinear}). The proof is complete.}
\end{proof}

The following key lemma will provide a way to obtain a special self-improvement of bounds for strongly Kreiss bounded operators. The proof is
a straightforward extension of \cite{AC}, where $X = L^p$ was considered. In order to obtain not too large explicit constant, some adjustment and optimization seemed necessary. Moreover, it can be helpful to see where the geometry of the space $X$ enters. 

We will use the notation $\sum_{a\leq m \leq b}$ for $a,b\in \R$ to denote the sum over all integers $m \in \Z$ such that $a\leq m\leq b$.

  \begin{lemma}\label{lem:key}
  Let $X$ be a $\UMD$ space and $p\in (1, \infty)$. Let $T\in \calL(X)$ be strongly Kreiss bounded with constant $K_s$. Suppose that there exists an increasing function $h:\mathbb{R}_+\to [1,\infty)$ such that for all $x\in X$ and $n\geq 2$,
\begin{align}\label{eq:esthN}
\Big\|\sum_{1\leq m\leq n}e_mT^m x\Big\|_{L^p(\mathbb{T};X)}\leq h(n)\|x\|.
\end{align}
Then there exists a constant $C_{p,X}>0$  such that for all $j\geq 0$, $n \geq 1$ and $x\in X$,
  \begin{align}\label{m5}
  \begin{aligned}
     \Big \|\sum_{n-\sqrt{n}+j\leq m\leq n}e_mT^m x\Big\|_{L^p(\mathbb{T};X)}& \leq K_s C_{p,X} h(\sqrt{n})\|x\|.
     \end{aligned}
  \end{align}
  \end{lemma}
  \begin{proof}
Define $S_n:=\sum_{1\leq m\leq \sqrt{n}}e_mT^m$. Then
     \begin{align*}
         \ee^{e_1 nT}S_n&=\sum_{k\ge 0}\frac{e_k(nT)^k}{k!}\sum_{1\leq m\leq \sqrt{n}}e_mT^m
         \\ & =\sum_{k\ge 0}\sum_{k+1\leq m\leq k+\sqrt{n}}\frac{n^k}{k!}e_mT^m \\
         &=\sum_{1\leq m\leq \sqrt{n}} \widetilde{b}_{n,m} e_mT^m +\sum_{m\geq \lfloor\sqrt{n}\rfloor+1}b_{n,m} e_mT^m,
     \end{align*}
where $\widetilde{b}_{n,m} := \sum_{0\leq k\leq m-1}\frac{n^k}{k!}$ and
$b_{n,m} := \sum_{m-\sqrt{n}\leq k\leq m-1}\frac{n^k}{k!}$.

We first consider the case $n\geq 6$ and thus $\sqrt{n} \geq 2$. Fix $j\geq 0$ and let $I_n = [n-\sqrt{n}+j,n]\cap \N$. Note that $$n-\sqrt{n}\geq \lfloor\sqrt{n}\rfloor+1.$$ By the boundedness of the Riesz projection with constant $R_{p,X}$, \eqref{k1} (which uses the strong Kreiss boundedness), and \eqref{eq:esthN} we obtain
 \begin{align}
\nonumber \Big\|\sum_{m\in I_n} b_{n,m} e_m T^m x \Big\|_{L^p(\mathbb{T};X)}
 &\leq R_{p,X}\|\ee^{e_1nT}S_nx\|_{L^p(\mathbb{T};X)}
 \\  &\leq K_s R_{p,X} e^n\|S_nx\|_{L^p(\mathbb{T};X)}\label{eq:estHsqrtNhelp}
 \\ & \leq K_s R_{p,X} e^n h(\sqrt{n})\|x\|.\nonumber
 \end{align}
Let $a_{n,m} := \ee^n b_{n,m}^{-1}$ for $m\in I_n$ and zero otherwise. Then by Lemma \ref{lem:factorial2}, $\|(a_{n,m})_{m \in \Z}\|_{\ell^\infty}\leq 32$ and $[(a_{n,m})_{m \in \Z}]_{V^1}\leq 978$. Therefore, the Fourier multiplier Lemma \ref{m16} and \eqref{eq:estHsqrtNhelp} imply that
\begin{align*}
\Big\|\sum_{m\in I_n}e^n e_m T^m x\Big\|_{L^p(\mathbb{T};X)} & \leq 1010 \, M_{p,X}\Big\|\sum_{m\in I_n} b_{n,m} e_m T^m x \Big\|_{L^p(\mathbb{T};X)}
\\ & \leq  1010\, M_{p,X} K_s R_{p,X} e^n h(\sqrt{n}) \|x\|.
\end{align*}
Dividing by $e^n$ gives \eqref{m5} with $$C_{p,X} := 1010 M_{p,X} R_{p,X}.$$

To prove the estimate for $n\leq 5$, note that by Lemma \ref{lemma:keyKreiss} we can write
\begin{align*}
\Big\|\sum_{m\in I_n}e_m T^m x\Big\|_{L^p(\mathbb{T};X)}& \leq R_{p,X} \Big\|\sum_{m=0}^n e_m T^m x\Big\|_{L^p(\mathbb{T};X)}
\\ & \leq 20(n+1) K_s R_{p,X}\|x\|\\&
\leq 1010 M_{p,X} K_s R_{p,X} h(\sqrt{n})\|x\|. \qedhere
\end{align*}
 \end{proof}

Combining Lemma \ref{lem:key} with the upper $\ell^q(L^p)$-decompositions, we obtain the following self-improvement result.
\begin{proposition}\label{p5}
Let $1<p,q<\infty$, $\gamma\in [0,1/q')$, and suppose that $X$ is a $\UMD$ which satisfies Theorem \ref{thm:main}\ref{it:main1} with $(q_0,\gamma_0)$ replaced by $(q,\gamma)$. Let $T\in \calL(X)$ be strongly Kreiss bounded with constant $K_s$. Suppose that there exist constants $d\in [0,1]$ and $P\geq 1$ such that for all $x\in X$ and $n\geq1$,
  \begin{align}
      \Big\|\sum^n_{m=1}e_m T^m x\Big\|_{L^p(\mathbb{T};X)}&\leq P n^{d}\|x\|.\label{6251}
  \end{align}
Then there is a constant $C_{p,X}'>0$ such that for all $x\in X$, and $n\geq 1$,
   \begin{align*}
     & \Big\|\sum^n_{m=1}e_m T^m x\Big\|_{L^p(\mathbb{T};X)}\leq P U  C_{p,X}' K_s n^{\frac12(d+\frac{1}{q}+\gamma)}\|x\|.
  \end{align*}
\end{proposition}
\begin{proof}
Let $N\in \N$ be such that $(N-1)^2<n\leq N^2$.
Then by the boundedness of the Riesz projection with constant $R_{p,X}$, the upper decompositions with constant $U$ and Lemma \ref{lem:key} with $h(n) = Pn^d$, we find
\begin{align*}
\Big\|\sum^{n}_{m=1}e_m T^m x\Big\|_{L^p(\mathbb{T};X)}^q
 &\leq R_{p,X}^q \Big\|\sum^{N^2}_{m=1}e_m T^m x\Big\|_{L^p(\mathbb{T};X)}^q
\\ &\leq  U^qR_{p,X}^q(2N)^{q\gamma}\sum^{N-1}_{k=0}\has{\Big\|\sum_{m=k^2+1}^{k^2+k}e_mT^m x\Big\|^{q}_{L^p(\mathbb{T};X)} \\&\hspace{2cm}+ \Big\|\sum_{m=k^2+k+1}^{(k+1)^2}e_mT^m x\Big\|^{q}_{L^p(\mathbb{T};X)}}
\\ &\leq P^q U^q R_{p,X}^q  C_{p,X}^q K_s^q (2N)^{q\gamma}\sum^{N-1}_{k=0} 2(k+1)^{dq} \|x\|^q
\\  &\leq 2^{1+q\gamma} P^qR_{p,X}^{q} U^q C_{p,X}^q K_s^q N^{(d+\gamma)q+1}\|x\|^q,
\end{align*}
where $C_{p,X}$ is the constant defined in the proof of Lemma \ref{lem:key}.
Since $N\leq 2\sqrt{n}$, this gives the result with constant (use $d+\frac2q+2\gamma<d+2\leq 3$)
\[2^{\frac1q+\gamma} R_{p,X} C_{p,X} 2^{d+\frac1q+\gamma}\leq 8\cdot1010 R_{p,X}^2 M_{p,X} = 8080\cdot R_{p,X}^2 M_{p,X}=:C_{p,X}'.\qedhere\]
\end{proof}

\subsection{Proof of Theorem \ref{thm:main}}
We can finally turn to the proof of the main result, which is an extension of the argument in \cite{AC}.

\begin{proof}[Proof of Theorem \ref{thm:main}]
Since $T$ and $T^*$ are both strongly Kreiss bounded, it follows from Lemma \ref{lemma:keyKreiss} and \eqref{eq:strongKreisslinear} that for $S \in \{T,T^*\}$ and $n\geq 1$ we have
 \begin{align*}
      \Big\|\sum^n_{m=1}e_m S^m x\Big\|_{L^p(\mathbb{T};X)}&\leq \min\cbraceb{20K_s (n+1)+1, K_s n \sqrt{2\pi (n+1)}} \|x\|\\&\leq
      21K_s n\|x\|,\label{6251}
  \end{align*}
using the first term in the minimum for $n> 64$ and the second term for $n\leq 64$. Therefore,  \eqref{6251} holds for
$T$ and $T^*$ with $d=c_0=d_0 := 1$ and $P:=21K_s.$

By a similar duality argument as in Proposition \ref{prop:dual}, one can check that the estimate
Theorem \ref{thm:main}\ref{it:main1}  holds with $(X, p, q_0, \gamma_0)$ replaced by $(X^*, p', q_1', \gamma_1)$. Define $c_{N}$ and $d_{N}$ for $N\geq 1$ by
\[c_{N}=\frac{1}{2^N q_0'} - \frac{\gamma_0}{2^N}, \qquad \text{and} \qquad d_{N}=\frac{1}{2^N q_1} -\frac{\gamma_1}{2^N}.\]
Let $F_{p,X} := U  C_{p,X}' K_s$ and $F_{p',X^*} = L C_{p',X^*}'K_s$, where $C_{p,X}'$ is the constant defined in the proof of Proposition \ref{p5}.
By Proposition \ref{p5} and an induction argument one sees that for every $N\geq 1$,
 \begin{align*}
      \Big\|\sum^n_{m=1}e_m T^m x\Big\|_{L^p(\mathbb{T};X)}&\leq P(F_{p,X})^N n^{c_{N}+\frac{1}{q_0}+\gamma_0}\|x\|, &&n\geq 1,\, x\in X,
      \\  \Big\|\sum^n_{m=1}e_m T^{*m} x^*\Big\|_{L^{p'}(\mathbb{T};X^*)}&\leq P(F_{p',X^*} )^N n^{d_{N}+\frac{1}{q_1'}+\gamma_1}\|x^*\|, &&n\geq 1,\, x^*\in X^*.
  \end{align*}

Let $n\geq 14$ and thus $n+2 \geq \ee^\ee$. We claim that there exist $N\in\N$ and $w_0,w_1>0$ such that
\begin{align}\label{eq:claimFpX}
(F_{p,X})^N n^{c_N}&\leq (\log (n+2))^{w_0}, \\ (F_{p',X^*})^N n^{d_N}&\leq (\log (n+2))^{w_1}\label{eq:claimFpX*}.
\end{align}
Indeed, let $N\in\mathbb{N}$ be such that $2^N<\frac{\log (n+2)}{\log(\log (n+2))}\leq 2^{N+1}$. Then
\[n^{c_N}\leq (n+2)^{c_N}
=(\log(n+2))^{\frac{\log (n+2)}{\log(\log (n+2))}c_N}\leq (\log (n+2))^{2/q_0' - 2\gamma_0}.\]
Moreover, from $2^N\leq \frac{\log (n+2)}{\log(\log (n+2))}$ and $\log(\log (n+2))\geq 1$, we obtain that $N\leq \frac{{\log(\log (n+2))}}{\log 2}$. Therefore, since $(F_{p,X})^N\ge 1$,
\[(F_{p,X})^N=e^{N \log F_{p,X}}\leq \ee^{\frac{\log(\log (n+2))}{\log 2} \log F_{p,X}}=(\log (n+2))^{\frac{\log F_{p,X}}{\log 2}}.\]
This gives \eqref{eq:claimFpX} with $w_0 = \frac{2}{q_0'}-2\gamma_0 + \frac{\log F_{p,X}}{\log 2}$. In the same way one sees that  \eqref{eq:claimFpX*} holds with $w_1 = \frac{2}{q_1}-2\gamma_1 + \frac{\log F_{p',X^*}}{\log 2}$.

From \eqref{eq:claimFpX} and \eqref{eq:claimFpX*}  we can conclude that for all $n\geq 14$, and $x\in X$,
 \begin{align}\label{m9}
  \Big\|\sum^n_{m=1}e_m T^m x\Big\|_{L^p(\mathbb{T};X)}&\leq P(\log (n+2))^{w_0} n^{\frac1{q_0}+\gamma_0}\|x\|,
  \\  \label{m10}
  \Big\|\sum^n_{m=1}e_m T^{*m} x^*\Big\|_{L^{p'}(\mathbb{T};X^*)} &\leq P(\log (n+2))^{w_1} n^{\frac1{q_1'}+\gamma_1}\|x^*\|.
   \end{align}
Since $\frac1{q_0}+\gamma_0\geq \frac12$ and $w_0\geq \frac{\log F_{p,X}}{\log 2}\geq \frac{\log(8080)}{\log(2)}$, one can readily check that \eqref{m9} extends to $n\leq 13$, where we used the bound \eqref{eq:strongKreisslinear} once more. The same holds for \eqref{m10}.

Applying \eqref{m10} with $n\geq 3$ replaced by $1+\sqrt{n}\leq 2\sqrt{n}$, we find
 \begin{equation}\label{m11}
      \Big\|\sum_{1\leq k\leq 1+\sqrt{n}}e_k T^{*k} x^*\Big\|_{L^{p'}(\mathbb{T};X^*)}\leq 2P (\log (n+2))^{w_1} n^{\frac{1}{2q_1'}+\frac{\gamma_1}{2}}\|x^*\|.
 \end{equation}
If $n\leq 2$,  (\ref{m11}) holds by \eqref{eq:strongKreisslinear}.
By Lemma \ref{lem:key} and \eqref{m9}, with $h(n):=P(\log (n+2))^{w_0} n^{\frac{1}{q_0}+\gamma_0}$, we obtain
\begin{align}\label{m12}
     \Big\|\sum_{n-\sqrt{n}\leq k\leq n}e_k T^{k} x\Big\|_{L^{p}(\mathbb{T};X)}\leq P K_s C_{p,X} (\log (n+2))^{w_0} n^{\frac{1}{2q_0}+\frac{\gamma_0}{2}}\|x\|.
     \end{align}
It follows that for all $n\geq 1$ and $x\in X, x^*\in X^*$,
\begin{align*}
    (1+\lfloor\sqrt{n}\rfloor)&\big|\lb x^*,T^{n+1}x\rb_{X^*,X}\big|\\&=\big|\sum_{1\leq k\leq 1+\sqrt{n}}\lb T^{*k}x^*,T^{n+1-k}x\rb_{X^*,X}\big|\\
    &=\Big|\int_{\mathbb{T}}\Big< \sum_{1\leq k\leq 1+\sqrt{n}}e_kT^{*k}x^*,\sum_{1\leq m\leq 1+\sqrt{n}}\bar{e}_mT^{n+1-m}x\Big>_{X^*,X}\dd t\Big|\\
    &\leq \Big\|\sum_{1\leq k\leq 1+\sqrt{n}}e_k T^{*k} x^*\Big\|_{L^{p'}(\mathbb{T};X^*)}\Big\|\sum_{n-\sqrt{n}\leq k\leq n}e_k T^{k} x\Big\|_{L^{p}(\mathbb{T};X)}\\
    &\leq 2 P^2 K_s C_{p,X} (\log (n+2))^{w_0+w_1} n^{\frac12(\frac{1}{q_0}+\gamma_0+\frac{1}{q_1'}+\gamma_1)}\|x\|\|x^*\|,
\end{align*}
where in the last step we used \eqref{m11} and \eqref{m12}.
Taking the supremum over $\|x\|\leq 1$ and $\|x^*\|\leq 1$, we obtain for all $n\geq 1$
\[\|T^{n}\|\leq C (\log (n+2))^{\beta} n^{\frac{1}{2}(\frac{1}{q_0}+\gamma_0-\frac{1}{q_1}+\gamma_1)},\]
where $C:=2 P^2 K_s C_{p,X}$ and $\beta = w_0+w_1$.
\end{proof}

\section{Results in Banach function spaces}\label{bfs}
In this final section we will discuss the particular case when $X$ is a Banach \emph{function} space. For details on Banach function spaces, the reader is referred to \cite{LT, Zaanen} and to the recent survey \cite{LNBfs}.

\begin{definition}
  Let $(S,\mathcal{A},\mu)$ be a $\sigma$-finite  measure space and denote the space of measurable functions $f \colon S \to \C$ by $L^0(S)$. A vector space $X \subseteq L^0(S)$ equipped with a norm $\norm{\,\cdot\,}$ is called a \emph{Banach function space over $S$} if it satisfies the following properties:
\begin{itemize}
  \item \emph{Ideal property:} If $f\in X$ and $g\in L^0(S)$ with $|g|\leq|f|$, then $g\in X$ with $\norm{g}\leq \norm{f}$.
  \item \emph{Fatou property:} If $0\leq f_n \uparrow f$ for $(f_n)_{n\geq 1}$ in $X$ and $\sup_{n\geq 1}\norm{f_n}<\infty$, then $f \in X$ and $\norm{f}=\sup_{n\geq 1}\norm{f_n}$.
  \item \emph{Saturation property:} For every measurable $E\subseteq S$ of positive measure, there exists a measurable $F\subseteq E$ of positive measure with $\one_F\in X$.
\end{itemize}
\end{definition}
We note that the saturation property is equivalent to the assumption that there is an $f \in X$ such that $f>0$ almost everywhere. Moreover, the Fatou property ensures that $X$ is complete.

We define the associate space $X'$ of a Banach function space $X$ as the space of all $g \in L^0(S)$ such that
\begin{equation*}
\norm{g}_{X'}:= \sup_{\norm{f}_X \leq 1} \int_{S} |{fg}|\dd\mu<\infty,
\end{equation*}
which is again a Banach function space. For $g\in X'$, define $\varphi_g:X\to \C$ by
\[\varphi_g(f):=\int_S fg \dd\mu,\]
which is a bounded linear functional on $X$, i.e., $\varphi_g\in X^*$. Hence, by identifying $g$ and $\varphi_g$, one can regard $X'$ as a closed subspace of $X^*$. Moreover, if $X$ is reflexive (or, more generally, \emph{order-continuous}), then $X'=X^*$.

The following notions, closely connected to type and cotype, will play an important role in this section (see \cite[Section 1.d]{LT} for details).
\begin{definition}
  $1 \leq p \leq q \leq \infty$. We call $X$ \emph{$p$-convex} if \begin{align*}
    \big\|(|f|^p+|g|^p)^{\frac{1}{p}}\big\|&\leq\big(\|f\|^p+\|g\|^p\big)^{\frac{1}{p}},&& f,g\in X,\intertext{and we call $X$ \emph{$q$-concave} if}
\big(\|f\|^q+\|g\|^q\big)^{\frac{1}{q}}&\leq\big\|(|f|^q+|g|^q)^{\frac{1}{q}}\big\|,&& f,g\in X.
\end{align*}
\end{definition}
Note that any Banach function space is $1$-convex by the triangle inequality and $\infty$-concave by the ideal property. One often defines $p$-convexity and $q$-concavity using finite sums of elements from $X$ and a  constant in the defining inequalities. However, by \cite[Proposition 1.d.8]{LT}, one can always renorm $X$ such that these constants are equal to one, yielding our definition. Moreover, $X$ is $p$-convex ($p$-concave) if and only if $X'$ is $p'$-concave ($p'$-convex).

For $s\in (0,\infty)$ and a Banach function space $X$, define $$X^s:=\{f\in L^0(S):|f|^{\frac{1}{s}}\in X\},$$ equipped with the quasi-norm $\|f\|_{X^s}:=\||f|^{\frac{1}{s}}\|^s_X$. If $s\leq 1$, $X^s$ is always a Banach function space. For $s>1$, $X^s$ is a Banach function space if and only if $X$ is $s$-convex. Note that for $0<s, p<\infty$, we have $(L^p(S))^s=L^{\frac{p}{s}}(S)$.

Our main result of Banach function spaces reads as follows.
\begin{theorem}\label{thm:mainbfs}
  Let $X$ be a Banach function space over $S$ and $s\in (1, 2)$. Suppose $X$ is $s$-convex and $s'$-concave, and
  $$
  X_{s}:= \bigl((X^s)'\bigr)^{\frac{1}{2-s}}
  $$
  is a $\UMD$ Banach function space.
Suppose that $T\in \calL(X)$ is strongly Kreiss bounded with constant $K_s$. Then there exist constants $C,\beta>0$ depending on $X$ and $K_s$ such that
\begin{equation*}
\|T^n\|\leq C n^{\frac{1}{2}-\frac{1}{s'}}(\log(n+2))^{\beta}, \qquad n \geq 1.
\end{equation*}
\end{theorem}
\begin{proof}
Note that $X^s$ is $\frac{s'}{s}$-concave, so $(X^s)'$ is $\frac{1}{2-s}$-convex and therefore we can conclude that $X_{s}$ is a well-defined Banach function space. By \cite[Corollary 2.12]{Ni23} and \cite{Ca64} we have
$$
X= ((X_s)')^{1-\frac{2}{s'}} \cdot L^{s'}(S)= \bigl[(X_s)',L^2(S)\bigr]_{\frac2{s'}}.
$$
Since $\frac{1-{2}/{s'}}{2} = \frac{1}{2}-\frac{1}{s'}$,
Corollary \ref{cor:HYcase} yields the result.
\end{proof}

Let us illustrate Theorem \ref{thm:mainbfs} and the space $X_{s}$ with some examples. We start by calculating the space $X_s$ for $X =L^p(S)$.

\begin{example}\label{example:Lp}
  Let $(S,\mathcal{A},\mu)$ be a $\sigma$-finite measure space and let $X=L^p(S)$ with $p\in (1, \infty)$. Let $1\leq s< \min\{p,p'\}$. Note that $L^p(S)$ is $s$-convex and $s'$-concave. Moreover, we have
  $$
  X_s = ((L^{\frac{p}{s}}(S))')^{\frac{1}{2-s}} = (L^{\frac{p}{p-s}}(S))^{\frac{1}{2-s}} = L^{\frac{(2-s)p}{p-s}}(S)=: L^{q}(S) .
  $$
  Since
  \begin{align*}
    q &= \frac{(2-s)p}{p-s} <\infty,
    \qquad \text{ and } \qquad
    q' = \frac{(2-s)p}{p+s-sp} = \frac{(2-s)p'}{p'-s} <\infty,
  \end{align*}
  we observe that $q \in (1,\infty)$ and thus that $X_s$ is a $\UMD$ Banach function space. Therefore,  Theorem \ref{thm:mainbfs} yields that for any strongly Kreiss bounded operator $T \in \calL(X)$, there are $C,\beta>0$ depending on $p$ and $K_s$ such that
\begin{equation*}
\|T^n\| = Cn^{\frac{1}{2}-\frac{1}{s'}} (\log(n+2))^\beta , \qquad n \geq1.
\end{equation*}
\end{example}
The above method can also be extended to non-commutative $L^p$-spaces.
Note that the result in Example \ref{example:Lp} is almost as sharp as Corollary \ref{cor:Lpcase}.
 So, in this particular case, the general result in Theorem \ref{thm:mainbfs} almost recovers the specialized result in  Corollary \ref{cor:Lpcase}. Of course, the advantage of Theorem \ref{thm:mainbfs} is that it is applicable to many other Banach function spaces, such as Lorentz, Orlicz and variable Lebesgue spaces. Let us illustrate the result for variable Lebesgue spaces:

\begin{example}
  Let $(S,\mathcal{A},\mu)$ be a $\sigma$-finite measure space, fix $p_0,p_1 \in (1,\infty)$  and assume $p \colon S\to [p_0,p_1]$ is measurable. Let  $X=L^{p(\cdot)}(S)$ be the space of all $f \in L^0(S)$ such that
  $$
  \int_S |f(x)|^{p(x)} \dd \mu(x)<\infty,
  $$
which, equipped with the corresponding Luxemburg  norm, is a Banach function space.
  Let $1\leq s< \min\{p_0,p_1'\}$ and note that $L^{p(\cdot)}(S)$ is $s$-convex and $s'$-concave. Moreover, by the same computation as in Example \ref{example:Lp}, we have
  $
  X_s = L^{q(\cdot)}(S) ,
  $
  where $q\colon S \to (1,\infty)$ satisfies
  \begin{align*}
    q(x) = \frac{(2-s)p(x)}{p(x)-s} < \frac{(2-s){p_1}}{p_0-s}<\infty,&&x \in S,\\
    q(x)' = \frac{(2-s)p(x)'}{p(x)'-s} < \frac{(2-s){p'_0}}{p_1'-s}<\infty,&&x \in S.
  \end{align*}
  So, by \cite[Corollary 1.2]{LVY18} we know that $X_s$ is a $\UMD$ Banach function space. Therefore,  Theorem \ref{thm:mainbfs} yields that for any strongly Kreiss bounded operator $T \in \calL(X)$, there are $C,\beta>0$ depending on $p$ and $K_s$ such that
\begin{equation*}
\|T^n\| = Cn^{\frac{1}{2}-\frac{1}{s'}} (\log(n+2))^\beta , \qquad n \geq 1.
\end{equation*}
\end{example}

 \subsection{Positive strongly Kreiss bounded operators}\label{positive}
We end this article by considering positive strongly Kreiss bounded operators $T$ (i.e., $Tf\ge0$ for all $f\ge 0$) on a Banach lattice. We refer to \cite{LT} for the definition of a Banach lattice and note that a Banach function space is a particular case of a Banach lattice.

The main result is the following extension of \cite{AC}, where the cases $X = L^p$ and $X$ is an AM or AL space were  considered.
\begin{theorem}\label{thm61}
Let $X$ be a Banach lattice. Suppose that $T\in \calL(X)$ is a positive operator which is strongly Kreiss bounded with constant $K_s$.
\begin{enumerate}[(1)]
\item\label{it:positive1} If $X$ is $p$-convex with $p\in (2,\infty]$, then there exist constants $C,\beta\ge 0$ depending on $X$ and $K_s$ such that
      $$\|T^n\|\leq C n^{\frac{1}{p}}(\log (n+2))^{\beta},\qquad  n\geq 1.$$
\item\label{it:positive2} If $X$ is $q$-concave with $q\in [1, 2)$, then there exist constants $C,\beta\ge 0$ depending on $X$ and $K_s$ such that
      $$\|T^n\|\leq C n^{\frac{1}{q'}}(\log (n+2))^{\beta},\qquad  n\geq 1.$$
\end{enumerate}
  \end{theorem}
  \begin{proof} The case $n=1$ is clear. In the following, we assume $n\geq 2$.

\ref{it:positive2}:\  Let $x\ge 0$. Using the Krivine calculus (see \cite[Proposition 1.d.1]{LT}),  Lemma \ref{lemma1}, the positivity of $T$, and  $\ell^1\hookrightarrow \ell^{q}$, we obtain
  \begin{align}\label{eq:Krivine}
\frac{\ee^n}{28\sqrt{n}}\Big(\sum_{n-\sqrt{n}\leq k\leq n}(T^k x)^{q}\Big)^{\frac{1}{q}}\leq \Big(\sum_{k\ge 0}\Big(\frac{n^k T^k x}{k!}\Big)^{q}\Big)^{\frac{1}{q}}\leq \sum_{k\ge 0}\frac{n^k T^k x}{k!}.
  \end{align}
  Since $X$ is $q$-concave, it follows from (\ref{k1}) that
  \begin{align*}
  \frac{\ee^n}{28\sqrt{n}}\big(\sum_{n-\sqrt{n}\leq k\leq n}\|T^k x\|^{q}\big)^{\frac{1}{q}}& \leq \frac{ \ee^n}{28\sqrt{n}}\Big\|\Big(\sum_{n-\sqrt{n}\leq k\leq n}(T^k x)^{q}\Big)^{\frac{1}{q}}\Big\|\nonumber\\
  &\leq \Big\| \sum_{k\ge 0}\frac{n^k T^k x}{k!}\Big\|=\|\ee^{nT}x\|\leq K_s\ee^n\|x\|,
  \end{align*}
    Therefore,
       \begin{align}\label{ps0}
    \Big(\sum_{n-\sqrt{n}\leq k\leq n}\|T^k x\|^{q}\Big)^{\frac{1}{q}}\leq 28K_s\sqrt{n}\|x\|.
    \end{align}
For all $x\in X,x^*\in X^*$, we can estimate
\begin{align*}
(1+\lfloor\sqrt{n}\rfloor)\big|\lb T^{n+1}x, x^*&\rb\big|^{q}=\sum_{1\leq k\leq 1+\sqrt{n}}\big|\lb T^{n+1-k}x,T^{*k}x^*\rb\big|^{q}\\
    &\leq\sum_{1\leq k\leq 1+\sqrt{n}}\|T^{n+1-k}x\|^{q} \|T^{*k}x^*\|^{q}\\
    &\leq \sum_{n-\sqrt{n}\leq k\leq n}\|T^{k} x\|^{q} \underset{1\leq k\leq 1+\sqrt{n}}{\sup}\|T^{*k}\|^{q}\|x^*\|^{q}.
\end{align*}
 Therefore, from \eqref{ps0}, taking the supremum over $\|x\|, \|x^*\|\leq 1$, and using $\|T^{n}\|=\|T^{*n}\|$, we find that
 \begin{align}\label{m17}
     \|T^{n}\|\leq 28K_s n^{\frac{1}{2q'}} \underset{1\leq k\leq2\sqrt{n}}{\sup}\|T^k\|.
 \end{align}
Since $\|T^n\|\leq K_s\sqrt{2\pi (n+1)}$  in any Banach space  by \eqref{eq:strongKreisslinear}, from \eqref{m17} we obtain
by induction that for any $N\geq 0$,
\begin{align*}
     \|T^{n}\|\leq 2\sqrt{\pi}K_s\cdot Q^N n^{\frac{1}{q'}+(\frac{1}{2}-\frac{1}{q'})2^{-N}},
 \end{align*}
where $Q=28 \cdot \sqrt{2}K_s$. Proceeding as in the proof of Theorem \ref{thm:main}, we see there exist $C, \beta>0$ such that
$$  \|T^{n}\|\leq C n^{\frac{1}{q'}}(\log (n+2))^{\beta}, \qquad  n\geq 2.$$

\ref{it:positive1}:\  We note that \ref{it:positive2} holds for $T^*$ on $X^*$ with exponent $1/p'$. Now \ref{it:positive1} follows by duality.
  \end{proof}

\section{Open problems}\label{sec:open}

In this section we collect some open problems related to the results of the paper.

The upper and lower decompositions imply (Fourier) type and cotype properties of $X$ as we have seen in Propositions \ref{prop:Utypecotype} and \ref{prop:Ltypecotype}. It would be interesting to know if a converse result holds.
\begin{problem}
Let $X$ be a $\UMD$ space and $p,q\in (1, \infty)$. Find a sufficient condition for $\ell^q(L^p)$-upper or lower decompositions in terms of (Fourier) type and cotype of the space $X$.
\end{problem}

Our decomposition properties are Fourier decomposition properties on $\T$. One may similarly define Fourier decomposition properties on $\R$, in which case it is natural to wonder if these properties would be equivalent. Note that transference methods are not directly applicable.

\begin{problem}
 Are the decomposition properties equivalent to their counterparts on $\R$?
\end{problem}

Even for the scalar field, we do not know for which $p$ and $q$ the upper and lower $\ell^q(L^p)$-decompositions hold. The following problem concerns the missing sharp endpoints.
\begin{problem}\label{prob:scalarfield}
Does the scalar field $\C$ have lower
$\ell^{p'}(L^p)$-decompositions for $p\in (1,2)$?
\end{problem}
If one reverses the roles of $\ell^{p'}$ and $L^p$, then the above estimate fails as was observed in \cite{CT}, which answered a problem left open in \cite{Rubio}.
In particular, a positive answer to Problem \ref{prob:scalarfield} would be a special case of the following:
\begin{problem}
Does Proposition \ref{prop:extrapolation-p} hold in the sharp case $\theta = \min\big\{\frac{p}r, \frac{p'}{r'}\big\}$?
\end{problem}

In Corollary \ref{cor:Lpcase} we have seen a sharp result for $X = L^p(S)$ for strongly Kreiss bounded operators. The method of Example \ref{example:Lp} can be adjusted to cover  non-commutative $L^p$-spaces, but yields a sub-optimal result.
\begin{problem}
Does Corollary \ref{cor:Lpcase} hold for non-commutative $L^p$-spaces?
\end{problem}

It seems that the bounds for positive operators obtained in Section \ref{positive} are non-optimal. Especially for $L^p(S)$-spaces we expect that there is an improvement. The bounds obtained from Corollary \ref{cor:Lpcase} and Theorem \ref{thm61} are different. Moreover, as observed in \cite{AC}, the bound of Theorem \ref{thm61} is worse than the one in Corollary \ref{cor:Lpcase} if $p\in (4/3, 4)$. It is unclear to us if and how positivity can help in the case $p\in (4/3, 4)$. Given the results for $L^1(S)$, $L^2(S)$ and $L^\infty(S)$ (see \cite{AC}), one could even hope that $\theta=0$ in the case of positive operators.

\begin{problem}
Let $T$ be a positive operator on $L^p(S)$ with $p\in (1, \infty)\setminus\{2\}$ which is strongly Kreiss bounded. What is the infimum of all $\theta\in [0,1/2)$ for which there exists a $C$ such that $\|T^n\|\leq C n^{\theta}$ for all $n\geq 1$.
\end{problem}

There has been a lot of interest in Kreiss bounded operators in finite dimensions (see \cite{kreiss19621,leveque1984,Spijker}). However, it seems to be unknown whether the obtained bounds in terms of the dimension can be improved for strongly Kreiss bounded operators.
\begin{problem}
Let $X$ be $d$-dimensional. Let $T$ be strongly Kreiss bounded. Determine the best $\theta\in (0,1]$ for which there exists a $C$ such that $\|T^n\|\leq C d^{\theta}$ for all $n\geq 1$.
\end{problem}

\appendix

\section{Technical estimate}
In this appendix we present some technical estimates based on the standard Stirling formula. These are quantified and optimized versions of results from \cite{AC}.
\begin{lemma}\label{lemma1}
    Let  $n\ge 2$. Then for all integers $k\in [0,2\sqrt{n}]$,
    \begin{align}\label{5.22}
       \frac{\ee^n}{28\sqrt{n}}\leq \frac{n^{n-k}}{(n-k)!}\leq \frac{\ee^n}{\sqrt{\frac{8\pi}{5} n}}.
    \end{align}
\end{lemma}
\begin{proof}
 The case $n\in [2,99]$ can be checked by hand. In the following we assume $n\geq 100$. It is elementary to check that $\log(1-x) \leq -\frac{2x}{2-x}$ for $x\in [0,1)$. Therefore, setting $g(x) = \frac{x^2}{2n-x}$, we find
\begin{align*}
\log\big(\ee^x (1-\tfrac{x}{n})^{n-x}\big)  &= x+ (n-x) \log(1-\tfrac{x}{n}) \\&\leq x- (n-x) \frac{2x}{2n-x}
 = g(x).
\end{align*}
The function $g\colon [0,2\sqrt{n}] \to [0,\infty)$ is increasing. It follows that for all $x\in [0,2\sqrt{n}]$,
\[g(x)\leq  g(2\sqrt{n}) = \frac{4n}{2n-2\sqrt{n}} = \frac{2}{1 - \frac1{\sqrt{n}}}\leq \frac{20}{9},\]
where in the last step we used $n\geq 100$. Therefore, we can conclude
\begin{align}\label{eq:elemenestexp}
\ee^{k} (1-\tfrac{k}{n})^{n-k} \leq \ee^{\frac{20}{9}}.
\end{align}
One can check that the latter is close to optimal. For $k =2\sqrt{n}$ and $n\to \infty$, the left-hand side tends to $\ee^{2}$.

Next, we show that \eqref{5.22} holds via the standard
Stirling formula
\begin{align}\label{eq:stirling}
e^{\frac{1}{12n+1}}\sqrt{2 \pi n}(\tfrac{n}{e})^n<n!<e^{\frac{1}{12n}}\sqrt{2 \pi n}(\tfrac{n}{e})^n, \qquad n\geq 1.
\end{align}

Let $n\geq 100$ and $k\in [0,2\sqrt{n}]$. It follows that
{$n-k\geq n-2\sqrt{n}\ge 80$.} Thus, by the upper estimate of \eqref{eq:stirling} and \eqref{eq:elemenestexp}, we have
\begin{align*}
(n-k)!&\leq \sqrt{2 \pi (n-k)}\big(\tfrac{n-k}{\ee}\big)^{n-k} \ee^{\frac{1}{12(n-k)}}
\\ &\leq  \sqrt{2 \pi n}\big(\tfrac{n}{\ee}\big)^{n-k}(1-\tfrac{k}{n})^{n-k}\ee^{\frac{1}{12\cdot 80}}
\\ &\leq \sqrt{2\pi}\sqrt{n}\big(\tfrac{n}{\ee}\big)^{n-k} \ee^{-k+\frac{20}{9}}\ee^{\frac{1}{960}}
\\ &= \ee^{\frac{20}{9}}\ee^{\frac{1}{960}}\sqrt{2\pi}\sqrt{n} n^{n-k} \ee^{-n}
\\ & \leq 28\sqrt{n} n^{n-k} \ee^{-n}.
    \end{align*}
 The first estimate in \eqref{5.22} is proved.

On the other hand, since $f(x):=\ee^{x} (1-\frac{x}{n})^{n-x}$ is increasing, we have $f(x)\geq f(0)=1$ for $x\in [0,2\sqrt{n}]$. Note that $n-k\geq n-2\sqrt{n}\ge\frac{4}{5}n$ due to $n\geq 100$. Thus the lower estimate of \eqref{eq:stirling} gives
\begin{align*}
(n-k)!&\ge \sqrt{2 \pi (n-k)}(\tfrac{n-k}{\ee})^{n-k} \ee^{\frac{1}{12(n-k)+1}}\\
&\ge  \sqrt{\frac{8\pi}{5}}\sqrt{n} n^{n-k} \ee^{-n} \ee^{k}(1-\tfrac{k}{n})^{n-k}\\
&\ge \sqrt{\frac{8\pi}{5}} \sqrt{n} n^{n-k}\ee^{-n},
 \end{align*}
finishing the proof.
\end{proof}

\begin{lemma}\label{lem:factorial2}
For $n\geq 2$ and $n-\sqrt{n}\leq m\leq n$, define
\begin{align*}
b_{n,m}&:=\sum_{m-\sqrt{n}\leq k\leq m-1}\frac{n^k}{k!},\\
a_{n,m} &:=\ee^n b_{n,m}^{-1},
\end{align*}
and $a_{n,m} = b_{n,m}=0$ otherwise. Then we have
$\|(a_{n,m})_{m\in \Z}\|_{\ell^\infty}\leq 32$ and $[(a_{n,m})_{m\in \Z}]_{V^1}\leq 978.$
\end{lemma}
\begin{proof}
The case $n\in [2,99]$ can be checked by hand. In the following we assume $n\geq 100$. We start with the boundedness of $a_{n,m}$ for $m\in [n-\sqrt{n},n]$.
From \eqref{5.22} it is almost immediate that \[|a_{n,m}|=\frac{e^n}{b_{n,m}}\leq \frac{28e^{n}}{\sum_{m-\sqrt{n}\leq k\leq m-1} \frac{e^n}{\sqrt{n}}} \leq 28\big(1+\frac{1}{\sqrt{n}-1}\big)\leq 32.\]

Next, we show that $(a_{n,m})_{m\in \Z}$ has bounded variation. First we fix $m\in[n-\sqrt{n},n-1]$ and let  $L:=\lceil m-\sqrt{n}\rceil$.
By \eqref{5.22},
\begin{align*}
    |a_{n,m+1}-a_{n,m}|&=\ee^n\big|b_{n,m+1}^{-1}-b_{n,m}^{-1}\big|
    =\ee^n\frac{|\frac{n^m}{m!}-\frac{n^L}{L!}|}{b_{n,m} b_{n,m+1}} \\ & \leq \ee^{-n} \bigl({\frac{n^m}{m!}+\frac{n^L}{L!}}\bigr)\cdot a_{n,m} a_{n,m+1}
     \leq \ee^{-n} \cdot  {2\frac{e^n}{\sqrt{\frac{8\pi}{5} n}}}\cdot 32^2 \leq \frac{914}{\sqrt{n}}.
\end{align*}
Therefore, we can conclude
   $$[(a_{n,m})_{m\in \Z}]_{V^1}\leq 2 \,\sup_{m\ge 1}|a_{n,m}|+\sum_{n-\sqrt{n}\leq m\leq n-1}|a_{n,m+1}-a_{n,m}|\leq 978,$$
   finishing the proof.
\end{proof}

\bibliographystyle{plain}
\bibliography{literature530}

\end{document}